\documentclass{amsart}
\usepackage{amssymb}
\usepackage{xcolor}   
\usepackage{mathrsfs}
\usepackage{tcolorbox}
\makeatletter
\def\l@subsection{\@tocline{2}{0pt}{2pc}{5pc}{}}
\makeatother
\usepackage[colorlinks=true, linkcolor=blue, citecolor=blue, urlcolor=blue]{hyperref}
\newtheorem{theorem}{Theorem}[section]
\newtheorem{lemma}[theorem]{Lemma}

\newtheorem{corollary}[theorem]{Corollary}
\theoremstyle{definition}
\newtheorem{definition}[theorem]{Definition}
\newtheorem{example}[theorem]{Example}

\theoremstyle{remark}
\newtheorem{remark}[theorem]{Remark}
\numberwithin{equation}{section}
\newcommand{\D}{\mathbb{D}}

\def\bigno{\bigskip \noindent}
\def\medno{\medskip \noindent}
\def\smallno{\smallskip \noindent}
\def\bignobf#1{\bigskip \noindent \textbf{#1}}
\def\mednobf#1{\medskip \noindent \textbf{#1}}

\def\bigno{\bigskip \noindent}
\def\medno{\medskip \noindent}
\def\smallno{\smallskip \noindent}
\def\bignobf#1{\bigskip \noindent \textbf{#1}}
\def\mednobf#1{\medskip \noindent \textbf{#1}}

\begin{document}

\title{Two Problems in Bergman Spaces with Non-radial Weights}
		
\author[X. Fang]{Xiang Fang}
	\address{Department of Applied Mathematics\\ National Yang Ming Chiao Tung University \\Hsinchu, Taiwan (R.O.C.) }
	\email{xfang@nycu.edu.tw}
	
	\author[F. Guo]{Feng Guo}
	\address{School of Mathematics \\ Nanjing University of Aeronautics and Astronautics\\ Nanjing 210016\\ P. R. China}  
\email{70207994@nuaa.edu.cn}

 	\author[S. Hou]{Shengzhao Hou}
 	\address{School of Mathematics Science\\ Soochow University\\ Suzhou 215006\\ P. R. China }
\email{shou@suda.edu.cn}
	
 		\author[Q. Zhou]{Qi Zhou}
 	\address{School of Mathematics Science\\ Soochow University\\ Suzhou 215006\\ P. R. China }
	\email{zhouqi@suda.edu.cn}

\dedicatory{}

\date{}

\begin{abstract}
This paper investigates two problems unified by the study of the uniform boundedness of the dilation operators (UBD)
\[
T_r f(z)=f(rz), \qquad 0<r<1,
\]
acting on weighted Bergman spaces \(A^p_\omega\) with   not necessarily radial weights.
We first characterize the random symbol space for   $A^p_\omega$   under a mild admissible condition (Theorem \ref{thm:main-nonradial}). This extends the main result of \cite{IMRN} from radial weights to non-radial weights. 
We then introduce two new notions, namely non-radial mixed norm spaces \(\mathcal{M}(p,q;\omega)\) and analytic tent spaces \(A(p,q;\omega)\), and we characterize their corresponding symbol spaces as well (Theorem \ref{thm:H-pq}, Theorem \ref{thm:tent}).
The novelty here is to employ a measure-disintegration framework to prove a general Littlewood-type theorem: $$(\mathcal{M}(p,q;\omega))_{\star} = H(2,q;\omega_\mathrm{r}),$$ from which the Bergman space result $(A^p_\omega)_\star = H(2,p;\omega_\mathrm{r})$ follows as the special case $p=q$. Among other things, UBD plays a pivotal role in the proofs of the preceding theorems.

\bigno The second main problem addressed in this paper is to establish a  local-to-global criterion for UBD, which remains largely unexplored for non-radial weights. 
 Our principle result in this part (Theorem \ref{thm:BHO}) asserts that UBD is guaranteed by two local geometric conditions: bounded hyperbolic oscillation (BHO) and a reverse-Carleson  tail condition (RC). This is the technical heart of the paper.
 
 \bigno 
In the course of our investigation, three  new types of problems arise naturally, each of independent interest:

 \begin{itemize}
     \item a two-weight top-maximal operator (Theorem \ref{thm:N-iff});
     \item a truncated maximal operator over hyperbolic balls (Theorem \ref{Hardy-Littlewood maiximal operator}); and
     \item a single-testing Carleson embedding problem (Theorem \ref{thm:WSMP-LD}).
 \end{itemize}

\end{abstract}

\maketitle

\newpage
\tableofcontents

\newpage
\section{Introduction and main results}

\noindent
Let \(\omega\) be a weight on the unit disk \(\mathbb{D}\subset \mathbb{C}\), and denote by \(A_\omega^{p}\) the corresponding weighted Bergman space for \(0<p<\infty\). When \(\omega\) is non-radial, both the function theory and the operator theory on \(A_\omega^{p}\) exhibit a remarkably rich structure  (see, e.g., \cite{Hu-Lv-Schuster-2019,Korhonen-Rattya-2019,Pelaez-Rattya-2024}), involving a delicate interplay between harmonic analysis, functional analysis, and complex analysis. Among the non-radial weights, a particularly significant class is the \(B_p\) weights (see, e.g., \cite{AlemanPottReguera19,Bekolle-Bonami-1978,Constantin-2010-Carleson})—introduced in analogy with the classical Muckenhoupt \(A_p\) condition for Euclidean spaces (see, e.g., \cite{Coifman-Fefferman-1974,Grafakos-Classical-3rd,Muckenhoupt-Wheeden-1973})—which provides an appropriate framework for developing
the \(L^{p}\)-theory of the Bergman projection on \(A_\omega^{p}\). Nonetheless, many other natural operators acting on these spaces remain poorly understood, and current techniques appear insufficient to treat them in a systematic way.

\bigskip

\noindent In this paper we investigate two problems that arise naturally in connection with a fundamental operator on weighted Bergman spaces, namely the family of dilation operators
\[
T_r f(z)=f(rz), \qquad 0<r<1.
\]
Although \(T_r\) is expected to be bounded in essentially all natural settings, there is currently no non-trivial quantitative description of its boundedness in the non-radial regime. This seems to stem more from the intrinsic complexity of non-radial weights than from any lack of significance  in the problem.
Moreover, \(T_r\) is closely tied to the radial maximal operator, whose boundedness with respect to non-radial weights remains an open problem. In the existing literature, indeed, the uniform boundedness  of \(T_r\) is often treated as a canonical assumption, and  little is known about which non-radial weights actually satisfy it.

\medskip

\noindent \textbf{Littlewood's theorem for random analytic functions.} Our first goal in this paper is to extend the Littlewood-type theorem of \cite{IMRN} for random Bergman functions to the non-radial setting. In \cite{IMRN} it was shown that the random symbol space associated with the weighted Bergman space \(L_a^p(\mathbb{D},dA_\alpha)\), where \(dA_\alpha=(1-|z|)^\alpha dA\), is a mixed norm space. More precisely,
\[
\big(L_a^p(\mathbb{D},dA_\alpha)\big)_\star = H(2,p;\tfrac{\alpha+1}{p}), \qquad 0<p<\infty.
\]
Our first main result generalizes this equality to weighted Bergman spaces \(A_\omega^p\) for a broad class of non-radial weights \(\omega\) satisfying a mild structural condition that we call \emph{admissibility}; see Definition \ref{Admissible-Weight}.

\medno 
Admissible weights are tightly connected to the uniform boundedness of the dilation operators \(\{T_r\}_{0<r<1}\) on \(A_\omega^p\). We undertake a detailed analysis of this problem and show that uniform boundedness is intertwined with a geometric regularity condition known as \emph{bounded hyperbolic oscillation} (BHO), together with a complementary notion that we call \emph{reverse Carleson}. The importance of BHO as a local regularity constraint was highlighted in the work of Aleman et al.\ \cite{AlemanConstantin09, AlemanPottReguera19, LimaniNicolau2023}, where it emerged as a central tool in the  theory of weighted analytic function spaces, particularly in connection with B{\`e}koll{\`e}–Bonami weights.

\medno We now turn to the precise formulation of these problems and the detailed statement of our first main result.

\newpage
\subsection{From Bergman  spaces to non-radial mixed norm spaces and tent spaces}

\smallno A classical theorem of Littlewood from the 1930's asserts that the random series \[\sum_{n\ge 0} \pm a_{n} z^{n}\] belongs to the Hardy space \(H^{p}(\mathbb{D})\) for every \(p \in (0,\infty)\) whenever the coefficients are square-summable, that is, whenever \(\sum_{n\ge0} a_{n} z^{n} \in H^{2}(\mathbb{D})\). Conversely, if the coefficients fail to be square-summable, then the random series admits no non-tangential boundary values and, in particular, does not lie in any Hardy space \cite{Littlewood1926, Littlewood1930, Paley1930B}. This circle of ideas has been systematically developed and presented in Kahane’s monograph \cite{Kahane1985}.

\medno 
To fix notation, we call a sequence of independent, identically distributed random variables a \emph{standard random sequence} if each variable is of one of the classical types: Bernoulli, Steinhaus, or Gaussian \(\mathcal{N}(0,1)\).
Let \(X=(X_{n})_{n\ge0}\) be a standard random sequence. For an analytic function \(f(z)=\sum_{n\ge0} a_{n} z^{n}\in H(\mathbb{D})\), we define its randomized Taylor series with respect to \(X\) by
\[
(\mathcal{R}_{X} f)(z) := \sum_{n\ge0} a_{n} X_{n} z^{n}.
\]

\begin{definition}[Random symbol space]
    The random symbol space of a (quasi-) Banach space $\mathcal{X}$ of analytic functions over the unit disk is
\[
(\mathcal{X})_{\star}:=\{f\in H(\mathbb{D}): \ \mathbb{P}(R_{X}f\in \mathcal{X})=1\}.
\]
\end{definition}

\medno Littlewood's theorem then  asserts that
\[(H^p(\mathbb{D}))_\star=H^2(\mathbb{D}), \qquad 0<p<\infty.\]

\medno Beyond the Hardy setting, this line of research was developed further by many authors, including Paley–Zygmund \cite{Paley1930A,Paley1930B}, Salem–Zygmund \cite{Salem1954}, Marcus–Pisier \cite{Marcus1978}, Anderson–Clunie–Pommerenke \cite{Anderson1974}, and Sledd \cite{Sledd1981A}. Their work established key criteria for determining when a random Taylor series belongs to various other classical analytic function spaces, such as \(H^\infty\), the Bloch space, and BMOA.
Among these spaces, obtaining a complete description of the random symbol space for BMOA has proven particularly difficult and remains an open problem. Recently, Nishry-Paquette \cite{Nishry2023} obtained new conditions ensuring that Gaussian analytic functions belong to BMOA. In a related direction, Konyagin-Queff{\`e}lec-Saksman-Seip \cite{Konyagin2022} established a sufficient criterion for random Dirichlet series to lie in BMOA.

\subsubsection{Weighted Bergman space}
\medno Our first main result for non-radially weighted Bergman spaces is the following: 

\begin{theorem}
\label{thm:main-nonradial}

Let $\omega$ be an admissible  weight on $\mathbb{D}$ (see Definition \ref{Admissible-Weight} below) and   $\{X_{n}\}$  
a standard random sequence. If $0<p<\infty$, then
\begin{equation}\label{eq:main-nonradial}
(A_{\omega}^{p})_{\star} = H(2,p;\omega_\mathrm{r}).
\end{equation}
\end{theorem}

\noindent Here we recall: 

\begin{definition}
    For $0<p<\infty$, the weighted Bergman space $A_{\omega}^{p}$ consists of all analytic functions $f$ on $\mathbb{D}$ such that
\[
\|f\| = (\int_{\mathbb{D}}|f(z)|^{p}\omega(z)dA(z))^{1/p}<\infty.
\]
\end{definition}

 \begin{definition}[The radial-circular decomposition]\label{Def:R-c-decomposition}
We disintegrate a measure $\omega$  based on its radial average:
\[
\overline{\omega}(s):=\frac{1}{2\pi}\int_{0}^{2\pi}\omega(se^{i\theta})d\theta.
\]
Then we define:
\begin{itemize}
    \item \emph{the radial projection measure} 
    $d\omega_\mathrm{r}(s):=\overline{\omega}(s)2sds$ on $(0, 1)$; and
    \item  \emph{the circular probability measures} $d\omega_{\mathrm{c}}(\theta;s):=\frac{\omega(se^{i\theta})}{\overline{\omega}(s)}\frac{d\theta}{2\pi}$ (for $\overline{\omega}(s)>0$).
\end{itemize} 
\end{definition}

\noindent For a radial weight $\omega$, we define the mixed norm spaces as follows: 

\begin{definition}[Radial mixed norm space]
For $0<p,q<\infty$, define $H(p,q;\omega)$ as the set of analytic functions $f$ on $\mathbb{D}$ such that
\[
\|f\|_{H(p,q;\omega)}^q=\int_{0}^{1} M_p^{\,q}(f,r)\,\omega(r)\,dr<\infty.
\]
\end{definition}

\noindent As the next step, now it is natural to attempt to define a mixed norm space \(H(p,q;\omega)\) for a non-radial weight \(\omega\) and, in view of \eqref{eq:main-nonradial}, to seek a characterization of \[(H(p,q;\omega))_{\star}.\] Such an extension, however, demands for caution. A naive definition would be
\[
H(p,q;\omega)
    =\left\{ f\in H(\mathbb{D}) : 
      \int_{0}^{1} 
      \left( \int_{0}^{2\pi} |f(re^{i\theta})|^{p}\,\omega(re^{i\theta})\,\frac{d\theta}{2\pi} \right)^{q/p}
      dr < \infty
    \right\}.
\]
But for general non-radial weights there is no principled justification for placing \(\omega\) inside the angular integral and it is awkward to keep a mere $dr$ outside. 
In summary, while extending weighted Bergman spaces to general weights is straightforward, the mixed norm setting is significantly more delicate, especially in the non-radial regime. This issue is addressed in the next subsection.

\subsubsection{Mixed norm spaces}

\begin{definition}[Non-radial mixed norm space]\label{Def:Non-Radial-Mixed-Norm-Space}
For $0<p,q<\infty$ and any weight $\omega$ on $\mathbb{D}$, we define the space $\mathcal{M}(p,q;\omega)$ to be the set of $f\in H(\mathbb{D})$ for which
\begin{equation}\label{eq:Non-Radial-Mixed-Norm-Space}
\|f\|_{\mathcal{M}(p,q;\omega)}^q=\int_{0}^{1}\left(\int_{0}^{2\pi}|f(se^{i\theta})|^{p}d\omega_{\mathrm{c}}(\theta;s)\right)^{q/p}d\omega_\mathrm{r}(s)<\infty.
\end{equation}
\end{definition}
\begin{remark}
    If $\omega$ is radial, then $\mathcal{M}(p,q;\omega)=H(p,q;\omega)$.
\end{remark}
 
\begin{theorem}%
\label{thm:H-pq}
Let $\omega$ be an admissible  weight on $\mathbb{D}$ (see Definition \ref{Admissible-Weight} below) and   $\{X_{n}\}$  a
standard random sequence. If $0<p,q<\infty$, then
\begin{equation}
(\mathcal{M}(p,q;\omega))_{\star} = H(2,q;\omega_\mathrm{r}).
\end{equation}
\end{theorem}
\medno In the special case $p=q$, Theorem~\ref{thm:H-pq} recovers Theorem~\ref{thm:main-nonradial}.

\subsubsection{Weighted analytic tent space}
\smallno In parallel to $\mathcal{M}(p,q;\omega)$, we 
introduce yet another generalization of $A^p_\omega$—namely, a class of weighted analytic tent spaces-motivated by a discrete analogue of the weighted Bergman norm, realized via the Whitney decomposition (Definition \ref{Def:Whitney-Decomposition}) of the unit disk $\mathbb{D}$.

\begin{definition}[Weighted tent space]\label{Def:Weighted-Tent-Space}
For $0<p, q<\infty$, the
weighted tent space $T(p,q;\omega)$ consists of the collection of measurable functions $g$ on  $\mathbb{D}$  for which there exists a Whitney decomposition $\{Q_j\}$ of $\mathbb{D}$ such that
\begin{equation}\label{eq:def-weighted-tent}
    \|g\|_{T(p,q;\omega)}^{q}
:=\sum_{j}(\int_{Q_{j}}|g(z)|^{p}\omega(z)dA(z))^{q/p}<\infty.
\end{equation}
\end{definition}

\begin{remark}
The well-definedness of the above definition is treated in the Appendix (see Subsection \ref{subsec:Equivalence-Whitney-Decompositions}).
\end{remark}

\begin{definition}[Weighted analytic tent space]
\[A(p,q;\omega)=T(p,q;\omega)\cap H(\mathbb D).\]
\end{definition}

\smallno The random symbol space of $A(p,q;\omega)$ is then characterized by the following:

\begin{theorem} \label{thm:tent}
Let $\omega$ be an admissible weight and   $\{X_{n}\}$   a standard random sequence. If $0<p,q<\infty$ and $f\in H(\mathbb{D})$, then  $\mathcal{R}_X f\in A(p,q;\omega)$ if and only if 
\[
\sum_{j}(\int_{Q_{j}}M_{2}(f,|z|)^{p}\omega(z)dA(z))^{q/p}<\infty.
\]
\end{theorem}

\noindent In particular, when $p=q$, Theorem \ref{thm:tent} reduces (equivalently) to Theorem \ref{thm:main-nonradial}.

\bigno Next we turn to the notion of UBD, which serves as the unifying theme of this paper. Its role in the proofs of the preceding theorems is to enable the use of the machinery developed in \cite{IMRN}; see, in particular, Lemma~\ref{lem:moment-criteria}.

\subsection{UBD: the analytic core of the admissible weight condition}

\subsubsection{UBD}
The admissible weight condition, which we require for our main theorems (Theorem \ref{thm:main-nonradial}, \ref{thm:H-pq}, \ref{thm:tent}), is fundamentally a statement about the analytic regularity of the space (see Definition \ref{Admissible-Weight} and Lemma~\ref{lem:Fspace-full}).  In fact, an admissible weight is one for which $A^p_\omega$ is well behaved in the sense that 
\begin{itemize}
    \item polynomials are dense in $A^p_\omega$, and
    \item  the dilation operators $\{T_r\}_{0<r<1}$ are uniformly bounded (UBD).
\end{itemize} 
That is, there exists $C>0$ such that
\[
\|T_{r}f\|\le C\|f\|\quad \text{for all } f\in A_{\omega}^{p} \text{ and } r\in(0,1).
\]
While this UBD property is analytically crucial, the literature has paid little attention to which general weights satisfy it, as far as we know. 

\medno The second problem we address in this paper is to establish sufficient criteria for UBD. In particular, we show that UBD is ensured by two local geometric properties of the weight: bounded hyperbolic oscillation and a reverse-Carleson tail condition.

\subsubsection{BHO and RC: a local criterion for UBD}
\begin{definition}[Bounded hyperbolic oscillation ]
For a fixed a hyperbolic radius $r_h>0$,  we say that $\omega$ has \emph{bounded hyperbolic oscillation} at scale $r_h$, if there exists an absolute constant $\Lambda$ such that 
\begin{equation}\label{eq:BHO}
\sup_{z\in\D}\ \sup_{w,w'\in B_{\rho_h}(z,r_h)} \frac{\omega(w)}{\omega(w')}\le \Lambda.
\end{equation}
A weight satisfying \eqref{eq:BHO} is said to be in $\mathrm{BHO}(r_h)$. %
\end{definition}

\begin{remark}
   The notion of BHO first appeared in \cite{AlemanPottReguera19}. Several other equivalent formulations  exist; some of these will appear in the sequel, particularly in the proofs of the main theorems.

\end{remark}

\begin{definition}[Carleson boxes, tops, and reverse--Carleson tail]\label{def:RC}
We identify the boundary $\partial\D$ with the torus $\mathbb{T} \cong [0, 1)$. For an arc $I\subset\partial\D$ with length $|I|$, the Carleson box and its top (thickness parameter $\delta \in(0,1)$ ) are
\[
S(I)=\left\{z \in \D:\ z/|z| \in I,\ 1-|I|< |z| <1\right\}\]
and 
\[T_\delta(I)=\left\{z \in \D:\ z/|z| \in I,\ 1-|I|< |z| <1-\delta|I|\right\}.
\]
For each integer $j\ge0$, let the $j$-th generation dyadic arcs on $\partial\mathbb{D}$ be
\[
I_{j,m}:=\left[m/2^{j},\ (m+1)/2^{j}\right),
\qquad 0\le m<2^j,
\]
with length $|I_{j,m}|=2^{-j}$. The collection of all such dyadic arcs is denoted by
\[
\mathcal{D}:=\left\{ I_{j,m}: j\ge 0,\ 0\leq m<2^j \right\}.
\]
The dyadic grid is the collection $\mathcal{Q} = \{Q_{j,m}\}_{I_{j,m} \in \mathcal{D}}$, where the dyadic box $Q_{j,m}$ is defined by
\[
Q_{j,m}:=\left\{z \in \D:\ z/|z|\in I_{j,m},\ 1-|I_{j,m}| \le |z| <1-|I_{j,m}|/2\right\}.
\]
Note that the boxes $\{Q_{j,m}\}$ are pairwise disjoint and form a partition of the annulus $\{z \in \D: 0 < |z| < 1\}$.

\medno We say that $\mu$ satisfies the \emph{reverse--Carleson tail condition} $\mathrm{RC}^\delta$ if for fixed $\delta \in (0,1)$, there exists  $\kappa>0$ such that for every dyadic arc $I\in \mathcal{D}$,
\begin{equation}\label{eq:RC}
\mu\big(S(I)\big)\ \le\ (\kappa+1)\,\mu\big(T_\delta(I)\big).
\end{equation}
\end{definition}

\medno Our main result concerning the second problem in this paper—namely, sufficient conditions for UBD—is the following:

\begin{theorem}[UBD from BHO + RC]\label{thm:BHO}
Let $0<p<\infty$, $r_h>0$ and $\mu=\omega\,dA$. If  $\omega\in\mathrm{BHO}(r_h)$ and $\mu$ satisfies $\mathrm{RC}^\delta$ for some $\delta\in(0,1)$, then UBD holds.
\end{theorem}

\medno It is natural to ask whether bounded hyperbolic oscillation alone already implies UBD.  BHO is, however, a purely local regularity condition and does not prevent the weight from losing mass along thin approach regions to the boundary, so it does not provide sufficient global control for the family $\{T_r\}_{0<r<1}$. The reverse--Carleson tail condition $\mathrm{RC}_\delta$ supplies exactly this missing global thickness by tying the mass of each Carleson box $S(I)$ to that of its top $T_\delta(I)$. In this regard, Theorem~\ref{thm:BHO} should be viewed as a genuinely joint consequence of BHO and $\mathrm{RC}_\delta$.

\subsection{Three auxiliary results}
\medno The remainder of this introduction is devoted to three further results that arose naturally in our analysis and are of independent interest. These questions are more purely harmonic-analytic in nature and may appeal to readers interested in that direction. We conclude by isolating three theorems corresponding to these auxiliary problems. They concern:

\begin{enumerate}
    \item a top maximal operator (Theorem \ref{thm:N-iff});
    \item a truncated maximal operator (Theorem \ref{Hardy-Littlewood maiximal operator}), and
    \item a single-testing Carleson embedding problem (Theorem \ref{thm:WSMP-LD}).
\end{enumerate}

\subsubsection{The top maximal operator}

\medno  The weak $(1,1)$ boundedness of the top maximal operator is a key step in proving Theorem~\ref{thm:BHO}. We establish this property by proving a more general characterization in a two-weight setting (see Theorem~\ref{thm:N-iff}).
Let us first clarify the definitions required for this two-weight setting. The first is the two-weight reverse Carleson tail condition, $\mathrm{RC}_{\mu\to\nu}^\delta$, which states that for fixed $\delta\in (0,1)$, there exists   $\kappa>0$ such that
\[
\nu\big(S(I)\big)\ \le\ (1+\kappa)\,\mu\big(T_\delta(I)\big)\quad\text{for all dyadic arcs }I\in \mathcal{D}.
\]
The second is the  top-maximal operator:
\[
\big(\mathcal{N}_{\mu}^\delta g\big)(z)\ :=\ \sup_{\substack{I\in\mathcal{D}\\ z\in S(I)}}\ \frac{1}{\mu\big(T_\delta(I)\big)}\int_{T_\delta(I)^\ast} |g|\,d\mu,
\]
where $T_\delta(I)^\ast$ is a thickening of $T_\delta(I)$ with bounded overlap, i.e., $ T_\delta(I)\subseteq T_\delta(I)^\ast$ and there exist $N>0$ such that, for $x\in \mathbb{D}$, $\sum_{I\in\mathcal{D}} \mathrm{1}_{T_\delta(I)^\ast}(x) \leq  C$.

\bigno
The operator $\mathcal{N}_{\mu}^\delta$ is a localized Hardy--Littlewood type maximal operator adapted to the Carleson geometry of $\D$.  
In particular, it is a genuinely two-weight object, and its weak $(1,1)$ boundedness from $L^1(\mu)$ to $L^{1,\infty}(\nu)$ encodes a balance between $\mu$ and $\nu$ in the spirit of classical two-weight results for maximal operators and Carleson embeddings (see for instance \cite{Gu92,KermanTorchinsky,LueckingRC} in the Euclidean, Hardy and Bergman settings). 
Our contribution here is that, in the present weighted Bergman setting, this balance is shown in Theorem~\ref{thm:N-iff} to be exactly equivalent to the reverse Carleson tail condition $\mathrm{RC}_{\mu\to\nu}^\delta$.

\begin{theorem} \label{thm:N-iff}
Let $\mu,\nu$ be two positive measures on $\D$ and $\delta \in (0,1)$. Then the condition 
$
\mathrm{RC}_{\mu\to\nu}^\delta$ holds if and only if  $\mathcal{N}_{\mu}^\delta$ is weak $(1,1)$.
\end{theorem}

\subsubsection{The truncated maximal operator}
 The truncated maximal operator 
is defined with respect to $0<r_1<r_0$ by
\[
M_{r_1,r_0}g(z):=\sup_{r_1<s\le r_0}\frac{1}{\mu\!\left(B_{\rho_h}(z,s)\right)}\int_{B_{\rho_h}(z,s)} |g(\zeta)|\,\mathrm{d}\mu(\zeta).
\]
for every locally integrable function $g$.

\bigno The operator $M_{r_1,r_0}$ differs substantially from the classical Hardy–Littlewood maximal operator in Euclidean analysis.  
It is \emph{truncated} in the hyperbolic radius parameter, being defined only for $0<r_1<r_0<r_h$, so that the supremum is taken over hyperbolic scales $s\in(r_1,r_0]$.  
This truncation is not merely technical: it is specifically designed to adapt the maximal operator to the intrinsic geometry of the hyperbolic disk.  
In the hyperbolic setting, the restriction to radii $s\in(r_1,r_0]$ is not merely a cosmetic truncation, but the key structural input that allows one to implement a Vitali-type selection procedure for hyperbolic balls; see Appendix~\ref{sec:Vitali of hyperbolic ball}. On this controlled range of scales, the geometry of $(\D,\rho_h)$ and the local behavior of $\mu$ are regular enough to carry out a hyperbolic Vitali covering argument, producing disjoint subfamilies and uniformly bounded overlap of dilated balls, which is the cornerstone of the weak $(1,1)$ estimate for $M_{r_1,r_0}$.

\bigno Under the bounded hyperbolic oscillation assumption on $\omega$, the measure $\mu=\omega\,dA$ is locally doubling with respect to $\rho_h$, and hence $(\D,\rho_h,\mu)$ behaves like a space of homogeneous type on small balls.
In this framework it is classical that Hardy–Littlewood type maximal operators associated with metric balls satisfy weak $(1,1)$ and strong $(p,p)$ estimates on $L^p(\mu)$, $1<p<\infty$ (see, for instance, \cite{DengHan-book} and the general theory of maximal operators on geometrically doubling metric measure spaces in \cite{Aldaz-avg}).
Theorem~\ref{Hardy-Littlewood maiximal operator} shows that, in our weighted Bergman setting, the same paradigm already applies at the level of the truncated maximal operator $M_{r_1,r_0}$: it enjoys the full scale of maximal-function type bounds under the single structural hypothesis $\omega\in\mathrm{BHO}(r_h)$.

\begin{theorem} \label{Hardy-Littlewood maiximal operator}
    Let $r_h>0$, $\mu=\omega\,dA$, $\omega\in\mathrm{BHO}(r_h)$ and $0<r_1<r_0$.
Then $M_{r_1,r_0}$ is bounded on $L^p(\mu)$ for $1<p<\infty$ and is of
weak type $(1,1)$.
\end{theorem}

\subsubsection{Single-testing Carleson embedding}

Finally, Theorem~\ref{thm:WSMP-LD} provides a single-testing criterion for the Carleson embedding
\[
A^p_\omega\hookrightarrow L^p(\nu)
\]
based on minimal local information about the underlying measures. Classical Carleson embeddings for Hardy and Bergman spaces (see \cite{Carleson-Orig,Hastings,Oleinik-Pavlov}) are usually formulated in terms of global Carleson measure conditions or testing on a rich family of reproducing kernels. In contrast, our result shows that, in the present weighted Bergman setting, it suffices to compare $\nu$ with a reference measure $\mu$ on dyadic boxes and to assume two purely local structural properties for $\mu$: a weighted sub-mean property and a local doubling condition. These hypotheses are automatically satisfied when $\mu=\omega\,dA$ with $\omega\in\mathrm{BHO}(r_h)$, but the proof of Theorem~\ref{thm:WSMP-LD} uses only the local assumptions themselves and does not rely on bounded hyperbolic oscillation or any global doubling. Recall that:

\begin{enumerate}
    \item A positive measure $\mu$ has the \emph{weighted sub-mean property}  if there exists $C_{\mathrm{ws}}$ such that for  locally integrable  subharmonic function $g\geq 0$,
    \begin{equation}\label{eq:WSMP}
    g(w)\ \le\ \frac{C_{\mathrm{ws}}}{\mu\!\big(B_{\rho_h}(w,r_0)\big)}\int_{B_{\rho_h}(w,r_0)} g\,d\mu, \qquad w \in \mathbb{D}.
    \end{equation}

    \item A measure $\mu$ is said to be \emph{local doubling}  if for each $\kappa \ge 1$, there exists $D_\kappa \ge 1$ such that
    \begin{equation}\label{eq:LD}
    \mu\!\big(B_{\rho_h}(z,\kappa r_0)\big)\ \le\ D_\kappa\ \mu\!\big(B_{\rho_h}(z,r_0)\big)\qquad\forall z\in\D.
    \end{equation}
\end{enumerate}

\begin{theorem}\label{thm:WSMP-LD}
Let $\mu$, $\nu$ be two positive measures and $r_0>0$. If the weighted sub-mean property and local doubling property hold for $\mu$ at the fixed scale $r_0$, and
\[
\nu(Q)\ \le\ C\,\mu(Q)\qquad\text{for every dyadic box }Q,
\]
then for every $f\in A^p_\omega$ and $0<p<\infty$, there exists a constant $K=K(r_0,C,C_{\mathrm{ws}})$  such that 
\begin{equation}\label{eq:CE}
\int_{\D} |f|^{p}\,d\nu\ \le K \int_{\D} |f|^{p}\,d\mu.
\end{equation}
\end{theorem}

\newpage
\subsection{Methodology overview}
\subsubsection{Outline of Theorem \ref{thm:H-pq} }
(1) Apply a \emph{measure–disintegration argument} to split the non-radial weight $\omega$ into its radial and angular components $\omega_r$ and $\omega_{\mathrm{c}}(\theta;r)$.
  (2) Represent $d\mu(z)=d\omega_{\mathrm{c}}(\theta;r) d\omega_r(r)$ to separate the $L^p$-norm into inner angular and outer radial integrals.
(3)  Define the random angular $L^p$-quantity 
 \eqref{Lp-inquality}. Using the Khintchine–Kahane inequality and some elementary inequalities to convert random $L^p$-moments into deterministic $L^2$-moments(\ref{eq:angular-moment-step1}).

\subsubsection{Outline of Theorem~\ref{thm:tent}}
(1) Define local measures $d\mu_j=\omega\,\mathbf{1}_{Q_j}\,dA$. Introduce the random local $L^p$-quantities
  \[
Y_j=\Big(\!\int_{Q_j}|\mathcal{R}_X f|^p\,d\mu_j\Big)^{1/p}.
  \]Apply the Khintchine–Kahane inequality to show 
  \(
  \mathbb{E}Y_j^p \asymp \int_{Q_j}M_2(f,|z|)^p\,d\mu_j.
  \)
(2) Use the vector-valued extension in the Banach ($p$-Banach) space $L^p(\mu_j)$ to obtain 
  \(\mathbb{E}Y_j^q \asymp_{p,q} (\mathbb{E}Y_j^p)^{q/p}.\)

\subsubsection{Outlines of Theorems~\ref{thm:BHO} and \ref{thm:N-iff}}
\begin{enumerate}
 
\item Under bounded hyperbolic oscillation (BHO) conditions,
each dyadic top $T_\delta(I)$ is enclosed between two hyperbolic balls of fixed
radii $r_{\mathrm{below}}(\delta)$ and $r_{\mathrm{above}}(\delta)$.  

\item A graph-coloring argument partitions the family $\{T_\delta(I)\}$ into finitely
many subcollections of disjoint tops, which ensures bounded overlap of their enlargements.

\item Propagating the BHO inequality along chains of intersecting balls yields uniform
measure comparability
\(
\mu(T_\delta(I))\!\le\!C_0\,\mu(B_{\rho_h}(w,r_0)).
\)

  \item A new class of dyadic top maximal operators is introduced in this paper.
This geometric framework allows the radial maximal operator to be dominated by its dyadic top counterpart.

\item Use a Calderón–Zygmund decomposition on maximal dyadic arcs to establish the weak $(1,1)$ boundedness of the dyadic top maximal operator.
\end{enumerate}

\subsubsection{Outline of Theorem~\ref{Hardy-Littlewood maiximal operator}}
(1) The proof uses the local doubling property implied by BHO.
(2)  For each level $\lambda>0$, a Vitali covering of hyperbolic balls yields disjoint
subfamilies whose 5-fold enlargements cover $E_\lambda$.
 (3)  Combine sublinearity and the $L^\infty$ bound with Marcinkiewicz interpolation to establish strong $(p,p)$ boundedness.

\subsubsection{Outline of  Theorem~\ref{thm:WSMP-LD}} 
(1) Use the local doubling condition together with the weighted sub-mean property to replace $B_{\rho_h}(z,r_0)$ by its  enlargement $Q^*$, yielding a localized boxwise estimate for $\sup_{z\in Q}|f(z)|^p$.
(2)  Summing over all dyadic cubes $Q$ and applying the bounded overlap property of $\{Q^*\}$ yields the global embedding
  \(
  \int_{\D}|f|^p\,d\nu
  \;\lesssim\;
  \int_{\D}|f|^p\,d\mu.
  \)

\subsection{Organization}

The remainder of the paper is organized as follows.  
Section~\ref{Sec:functional-framework-symbol-spaces} establishes preliminary issues for symbol spaces.  
Section~\ref{sec:Proof-Theorem-1.8} presents the proof of Theorem~\ref{thm:H-pq},  
and Section~\ref{Sec:proof-tent} proves Theorem~\ref{thm:tent}.  
Section~\ref{N-iff} proves Theorem~\ref{thm:N-iff}.  
Section~\ref{UBD}  deduces Theorem~\ref{thm:BHO}.  
Section~\ref{Hardy-maximal} establishes Theorem~\ref{Hardy-Littlewood maiximal operator}.  
Section \ref{Sec:Carleson-embedding-equivalence} proves  Theorem~\ref{thm:WSMP-LD}.  
An Appendix collects auxiliary lemmas and technical results used in the proofs.

\subsection{Notations}
\begin{itemize}
    \item $\mathbb{D}=\{z\in\mathbb{C}:|z|<1\}$ is the open unit disk in the complex plane.
    \item $B_{\rho_h}(z,r)$ denotes the hyperbolic ball with center $z$ and radius $r$.
    \item $dA(z)=\frac{1}{\pi}\, r\,dr\,d\theta$ denotes the normalized
Lebesgue area measure on $\mathbb{D}$.
\item $A(B_{\rho_h}(z,r))=\int_{B_{\rho_h}(z,r)} d A$ denotes the Euclidean area of a hyperbolic ball.
\item A weight $\omega$ is a positive, measurable function on $\mathbb{D}$.
\item The associated weighted measure is $d\mu(z) = \omega(z)\,dA(z)$.
\item $\mu(Q)=\int_{Q}d\mu $ denotes the measure of the set $Q$.
\item $H(\mathbb{D})$ denotes the space of  analytic functions on the unit disk $\mathbb{D}$.
\item For $0<p<\infty$, $0<r<1$  and $f\in H(\mathbb{D})$, we write
\(M_{p}(f,r)=\left(\int_{0}^{2\pi} |f(re^{i\theta})|^p \frac{d\theta}{2\pi}\right)^{1/p}.\)
\item Denote by $X\lesssim_{a,b}Y$ if $X\le C\,Y$ with $C$ depending only on the displayed parameters; $X\asymp Y$ means both $X\lesssim Y$ and $Y\lesssim X$.
\end{itemize}

\section{Preliminary  issues for random symbol spaces}\label{Sec:functional-framework-symbol-spaces}

\noindent The goal of this section is to build the functional-analytic framework required for Theorems \ref{thm:main-nonradial}, \ref{thm:H-pq} and \ref{thm:tent}.

\subsection{Admissible weight}\label{sec:Admissable-Weight}

\noindent To place the weighted Bergman spaces $A^p_\omega$ in a workable functional--analytic setting, we assume  two mild hypotheses stated next---finite mass and bounded point evaluations. Under these assumptions $A^p_\omega$ is an F-space (Banach if $p\ge1$) with all analytic polynomials contained; see Lemma~\ref{lem:Fspace-full}.
\begin{definition}
We assume that the weight $\omega$ satisfies:
\begin{enumerate}
    \item Finite mass :$\int_{\mathbb{D}}\omega(z)dA(z)<\infty$; and 
    \item Bounded point evaluations: for every $z\in\mathbb{D}$, the functional $L_{z}(f)=f(z)$ is
    continuous on $A_{\omega}^{p}$ and for each compact $K, \sup_{z\in K}\|L_{z}\|<\infty$.
\end{enumerate}
\end{definition}

\medno We further impose   a  mild condition on the (possibly non-radial) weight—a condition we call admissibility. Under admissibility we can describe the symbol space of the weighted Bergman space \(A^p_\omega\) and, subsequently, its mixed-norm and weighted analytic tent counterparts (Theorems \ref{thm:main-nonradial}, \ref{thm:H-pq} and \ref{thm:tent}).

\begin{definition} \label{Admissible-Weight}
    We call a weight $\omega$  on $\mathbb{D} \subseteq \mathbb{C}$ \emph{admissible}  if it satisfies finite mass, bounded point evaluations, and 
\begin{equation}\label{eq:DCP}
\lim_{r\rightarrow1^{-}}\|f_{r}-f\|=0 \quad \text{ for every }f\in A_{\omega}^{p}.
\end{equation}

\end{definition}

\medno In terms of the dilation operators \(T_r\)  \((0<r<1)\), that is, 
\(
(T_r f)(z)=f_r(z):=f(rz),
\)
\eqref{eq:DCP} asserts that $T_{r}\rightarrow I$ strongly; we refer to this as the dilation convergence property.
Beside this, we single out two properties of the operators \(T_r\) and establish precise implication relations among them; see Lemma~\ref{Fundamental-Relationship}.

\begin{definition}
We consider the following properties for $A_{\omega}^{p}$:
\begin{enumerate}
    \item Dilation convergence property: for every $f\in A_{\omega}^{p}$, $\lim_{r\rightarrow1^{-}}\|f_{r}-f\|=0$.
    \item Polynomial density property : the set of   polynomials P is dense in $A_{\omega}^{p}.$
    \item Uniform boundedness of dilations (UBD): the operators $\{T_{r}\}_{0<r<1}$ are uniformly
    bounded; i.e., there exists $C>0$ such that $\|T_{r}f\|\le C\|f\|$ for all $f\in A_{\omega}^{p}$ and $r\in(0,1)$.
\end{enumerate}
\end{definition}

\begin{example}
We list a few representative families to indicate the abundance of admissible non-radial weights; details are routine and omitted.
\begin{enumerate}
\item Separable angular–radial weights. \(\omega(z)=v(|z|)\,h(\arg z)\) with \(v\) any admissible radial weight (e.g., doubling/regular) and \(h\in L^{1}(\mathbb{T})\) bounded above and below away from \(0\). Then \(\omega\) is admissible.

\item Small angular perturbations. \(\omega(z)=v(|z|)\,e^{\psi(\arg z)}\) where \(v\) is admissible radial and \(\psi\in \mathrm{VMO}(\mathbb{T})\cap L^\infty(\mathbb{T})\). Since \(e^{\psi}\) is bounded above/below, \(\omega\) is \(L^\infty\)-equivalent to \(v\) and hence admissible.

\item Piecewise-constant angular factors. \(\omega(z)=v(|z|)\,h(\arg z)\) with admissible radial \(v\) and \(h(\theta)=1+\eta\,\mathbf{1}_{I}(\theta)\) for an arc \(I\subseteq\mathbb{T}\) and \(0<\eta<1\). Here \(h\) is bounded above/below, so \(\omega\) is admissible.
\end{enumerate}

\end{example}

\subsection{A framework for the boundedness of dilation operators}\label{subsec:boundedness-dilation-operator}
A satisfactory framework for this setting exists and is largely familiar to experts. In particular, the dilation convergence property is equivalent to the combination of polynomial density and the uniform boundedness of dilations (UBD); see Lemma~\ref{Fundamental-Relationship} below. This fact, along with the next few lemmas, is  known to specialists, and we include brief proof outlines for completeness.

\begin{lemma}\label{Fundamental-Relationship}
Let $0<p<\infty,$ and we assume the   conditions of finite mass and bounded point evaluations. Then the dilation convergence property  holds if and only if both the polynomial density property  and the uniform boundedness of dilations (UBD) hold.
\end{lemma}

\begin{proof}

The implication from the dilation convergence property to UBD is proved in Lemma~\ref{lem:DCP-implies-UBD}, while the implication from the dilation convergence property to the polynomial density property is established in Lemma~\ref{lem:DCP-implies-PDP}. Conversely, the conjunction of the polynomial density property with UBD yields the dilation convergence property, as shown in Lemma~\ref{lem:PDPUBD-implies-DCP}.
\end{proof}

\begin{lemma}\label{lem:Fspace-full}
Let $0<p<\infty$. The space $A^p_\omega$ is a complete metrizable topological vector space (hence an F-space, and in particular a Banach space when $p\ge1$) if and only if bounded point evaluation property  holds. In particular, under the bounded point evaluation property, the canonical embedding $A^p_\omega\hookrightarrow L^p(\mu)$ is closed, hence $A^p_\omega$ is a closed subspace of $L^p(\mu)$.
\end{lemma}

\begin{proof}
Let $(f_n)$ be Cauchy in $A^p_\omega$. Then 
\(
\sup_{z\in K}|f_n(z)-f_m(z)|\le C_K\|f_n-f_m\|_{A^p_\omega}\xrightarrow[n,m\to\infty]{}0
\) for every compact $K$.
Hence, $(f_n)$ is uniformly Cauchy on each compact subset, and thus converges locally uniformly to some $h\in H(\mathbb{D})$.
Then the  completeness of $A^p_\omega$ follows from the fact that $L^p(\mu)$ is complete. 
Conversely, by the closed graph theorem for $F$-spaces, 
 the identity map
$
I:A^p_\omega\longrightarrow H(\mathbb{D}),
$
is continuous.
For each fixed $z\in\mathbb{D}$, the evaluation $\mathrm{ev}_z:H(\mathbb{D})\to\mathbb{C}$ is continuous, hence $L_z=\mathrm{ev}_z\circ I$ is continuous on $A^p_\omega$, which yields the first part of bounded point evaluation property. Since for each $f\in A^p_\omega$ one has $\sup_{z\in K}|L_z(f)|=\sup_{z\in K}|f(z)|<\infty$, the uniform boundedness principle for F-spaces implies $\sup_{z\in K}\|L_z\|<\infty$, which establishes the second part of bounded point evaluation property.
\end{proof}

\begin{lemma}\label{lem:DCP-implies-UBD}
Let $0<p<\infty$ and we assume the conditions of finite mass and bounded point evaluations. 
If the dilation convergence property holds on $A^p_\omega$, then the dilations 
$\{T_r\}_{0<r<1}$ are uniformly bounded on $A^p_\omega$; that is,
\[
\sup_{0<r<1}\,\|T_r\| \;<\; \infty.
\]
\end{lemma}

\begin{proof}
If the dilation convergence property holds, then $T_r f \to f$ for every $f$. A convergent sequence is bounded, so the family $\{T_r\}$ is pointwise bounded. Since $A^p_\omega$ is an F-space (Lemma \ref{lem:Fspace-full}), the uniform boundedness theorem for F-spaces implies $\{T_r\}$ is uniformly bounded. Then, UBD holds.
\end{proof}

\begin{lemma}\label{lem:DCP-implies-PDP}
Let $0<p<\infty$ and we assume only the finite mass property. If the dilation convergence property  holds on $A^p_\omega$,
then the polynomial density property  holds.
\end{lemma}

\begin{proof}

Let $f \in A^p_\omega$ and $\epsilon > 0$. By dilation convergence property, choose $r$ such that $d(f, f_r) < \epsilon/2$, where $d$ is the metric in $A_\omega^p$.
The function $f_r(z) = f(rz)$ is analytic on $\bar{\D}$, so its Taylor series converges uniformly to $f_r$. We can choose a Taylor polynomial $P$ such that $\sup_{z \in \D} |f_r(z) - P(z)| < \delta$.
Let $W = \int \omega dA<+\infty$. We choose $\delta$ such that $\delta^p W < \epsilon/2$. Then
$ d(f_r, P) <\epsilon/2. $
By the triangle inequality for the metric $d$, we have 
$ d(f, P)< \epsilon.$
Thus, polynomial density property holds.
\end{proof}

\begin{lemma}\label{lem:PDPUBD-implies-DCP}
Let $0<p<\infty$. Assume that on $A^p_\omega$ the polynomial density property  holds and the dilation operators are uniformly bounded.
Then the dilation convergence property  holds:
\(
\lim_{r\uparrow 1}\,\|T_r f-f\|_{A^p_\omega}=0\ \text{for every }f\in A^p_\omega.
\)
\end{lemma}

\begin{proof}
 Let $P(z)=\sum_{n=0}^N a_n z^n$ be a polynomial. Then for $1<p<\infty$, 
$
\|P-P_r\|_{A^p_\omega}
\xrightarrow[r\uparrow1]{}0.
$
Thus $T_rP\to P$ in $A^p_\omega$ for every polynomial $P$.
Let $f\in A^p_\omega$ and $\varepsilon>0$. If $p\ge1$, by the polynomial density property, there exists a polynomial $P$ such that
$
\|f-P\|_{A^p_\omega}<\frac{\varepsilon}{2(C+1)},
$ where $C=\sup_{0<r<1}\|T_r\|$.
Moreover, choose $r_0\in(0,1)$ so that $\|P-P_r\|_{A^p_\omega}<\varepsilon/2$ for all $r>r_0$.
Then, for $r>r_0$
\[
\|f-T_r f\|_{A^p_\omega}
\le \|f-P\|_{A^p_\omega}
+ \|P-P_r\|_{A^p_\omega}
+ \|T_r(P-f)\|_{A^p_\omega}
< \tfrac{\varepsilon}{2(C+1)}+\tfrac{\varepsilon}{2}+C\tfrac{\varepsilon}{2(C+1)}
=\varepsilon.
\]
If $0<p<1$, an analogous argument works with metric $\|\cdot\|^p_{A^p_\omega}$. Consequently, dilation convergence property holds for all $0<p<\infty$.
\end{proof}

\subsection{The stochastic effect of the UBD condition}\label{subsec:stochastic-effect-UBD-condition}

In this section, we lay the groundwork for the stochastic arguments. We show that within admissible F-space framework, the   almost sure membership needed for the random symbol space is equivalent to the more tractable condition of having finite moments (Lemma \ref{lem:moment-criteria}). This equivalence is the   bridge that allows us to compute the moment identities in the proofs of our theorems and it is enabled by our UBD assumption, hence the full force of the machinery in \cite{IMRN} comes into  play.

\begin{lemma}\label{lem:moment-criteria}

Let $\omega$ be admissible,  $(X_n)_{n\ge0}$ be a standard random sequence and  $0<p<\infty$.
For $f(z)=\sum_{n\ge0} a_n z^n\in H(\mathbb{D})$, write
\[
S_N:=\sum_{n=0}^{N} a_n X_n\,z^n.
\]
 The following are equivalent:
\begin{enumerate}
\item $\mathcal{R}_X f\in A^p_\omega$ almost surely;
\item $\displaystyle \sup_{N\ge0}\,\|S_N\|_{A^p_\omega}<\infty$ almost surely;
\item $\mathbb{E}\,\|\mathcal{R}_X f\|_{A^p_\omega}^{\,t}<\infty$ for some $t>0$;
\item $\mathbb{E}\,\|\mathcal{R}_X f\|_{A^p_\omega}^{\,t}<\infty$ for every $t>0$.
\end{enumerate}
Moreover, for each $t>0$ the quasi‑norms $\big(\mathbb{E}\,\|\mathcal{R}_X f\|_{A^p_\omega}^{\,t}\big)^{1/t}$ are mutually equivalent
(as $t$ varies), and for each fixed $t$ they are equivalent across the three standard randomizations, with constants depending only on $t$.
\end{lemma}

\begin{proof}
    By the admissible weight assumptions, if $f \in A_\omega^p$, then $f_r \rightarrow f$ in $A_\omega^p$. The remaining discussion is parallel to that of Theorem 8 in \cite{IMRN} and is therefore omitted.
\end{proof}

\section{Proof of Theorem~\ref{thm:H-pq}}\label{sec:Proof-Theorem-1.8}

\begin{proof}
For $f(z)=\sum_{n\ge0} a_n z^n\in H(\D)$, we define the  randomized
Taylor series at radius $r$ by
\[
S_r
  := \sum_{n\ge0} a_n X_n\, r^n e^{in\theta},
\qquad 0\le\theta<2\pi.
\]
The angular $L^p$--norm at radius $r$ (with respect to the disintegrated
circular probability measure $d\omega_{\mathrm{c}}(\theta;r)$) is then given by
\begin{equation}\label{Lp-inquality}
   Z_r
  := \Bigg(\frac1{2\pi}\int_0^{2\pi} |S_r|^p\,d\omega_{\mathrm{c}}(\theta;r)\Bigg)^{\!1/p}. 
\end{equation}
We claim that there exist constants $0<c_{p,q}\le C_{p,q}<\infty$  such that
\begin{equation}\label{eq:angular-moment-step1}
c_{p,q}\,\Bigg(\sum_{n\ge0} |a_n|^2 r^{2n}\Bigg)^{q/2}
\ \le\ 
\mathbb{E}\big[Z_r^q\big]
\ \le\
C_{p,q}\,\Bigg(\sum_{n\ge0} |a_n|^2 r^{2n}\Bigg)^{q/2}.
\end{equation}

\smallno We treat separately the cases $q\ge p$ and $q<p$.

\medskip
\noindent\emph{Upper bound in \eqref{eq:angular-moment-step1}: case $q\ge p$.}
Consider the convex function $\varphi(t)=t^{q/p}$ on $[0,\infty)$.
Applying Jensen's inequality gives
\[
\mathbb{E}\,Z_r^q
  \;=\; \mathbb{E} \varphi\!\left(\int |S_r|^p\,d\omega_{\mathrm{c}}(\theta;r)\right)
  \;\le\;  \int \mathbb{E}\,|S_r|^q\,d\omega_{\mathrm{c}}(\theta;r).
\]
By the (scalar) Khintchine--Kahane inequality (for the chosen standard model),
we have
\[
\mathbb{E}\,|S_r|^q
  \;\asymp_{p,q}\;
  \Bigg(\sum_{n\ge0} |a_n|^2 r^{2n}\Bigg)^{q/2},
\]
with equivalence constants independent of $r$, $\theta$ and the particular choice
of standard randomization. 

\medskip
\noindent\emph{Upper bound in \eqref{eq:angular-moment-step1}: case $q<p$.}
By definition, the H\"older's inequality  and 
 again the Khintchine--Kahane inequality (now with exponent $p$) gives
\[
\mathbb{E}\, Z_r^{\,q}\ \le\  \left( \mathbb{E}\, Z_r^{\,p}\right)^{q/p} \ 
=\left(\int \mathbb{E}|S_r|^{p}\, d\omega_{\mathrm{c}}(\theta;r) \right)^{q/p}
\asymp \Big(\sum_{n\ge0} |a_n|^2 r^{2n}\Big)^{q/2}.
\]
This completes the upper bound in \eqref{eq:angular-moment-step1}.

\medskip
\noindent\emph{Lower bound in \eqref{eq:angular-moment-step1}: case $q\ge p$.}
Since $q\ge p$, the elementary inequality
\[
\Big(\int |S_r|^p\,d\omega_{\mathrm{c}}(\theta;r)\Big)^{q/p}
   \le \int |S_r|^q\,d\omega_{\mathrm{c}}(\theta;r)
\]
holds, which implies $Z_r^q\le \int |S_r|^q\,d\omega_{\mathrm{c}}(\theta;r)$. Taking expectations and using  the  Khintchine--Kahane inequality  yields the desired lower bound.

\medskip
\noindent\emph{Lower bound in \eqref{eq:angular-moment-step1}: case $q<p$.}
In this case, the function $\varphi(t)=t^{q/p}$ is concave on $[0,\infty)$.
Applying Jensen's inequality in the reverse direction yields
\[
\mathbb{E}\,Z_r^q
  = \mathbb{E}\,\varphi\!\left(\int |S_r|^p\,d\omega_{\mathrm{c}}(\theta;r)\right)
  \ge \Bigg(\int \mathbb{E}\,|S_r|^p\,d\omega_{\mathrm{c}}(\theta;r)\Bigg)^{\frac{q}{p}} \gtrsim_{p,q}
  \Bigg(\sum_{n\ge0} |a_n|^2 r^{2n}\Bigg)^{\frac{q}{2}},
\]
which proves the lower bound in the remaining case.

\medno 
Combining all four cases yields \eqref{eq:angular-moment-step1}.

\medno We now pass from the one--radius estimate \eqref{eq:angular-moment-step1} to
a full mixed--norm identity in the radial variable. 
Note that
\[
\|\mathcal{R}_X f\|_{\mathcal{M}(p,q;\omega)}^q
  = \int_0^1 Z_r^q\,d\omega_\mathrm{r}(r).
\]
Taking expectations and applying the point-wise two-sided estimate \eqref{eq:angular-moment-step1} for
each fixed $r\in(0,1)$, we obtain,
\begin{equation}\label{eq:global-moment-start}
\mathbb{E}\,\|\mathcal{R}_X f\|_{\mathcal{M}(p,q;\omega)}^q
  = \int_0^1 \mathbb{E}\big[Z_r^q\big]\,d\omega_\mathrm{r}(r)\asymp \|f\|_{H(2,q;\omega_\mathrm{r})}^q
\end{equation}
with equivalence constants independent of $f$.
 
\medno 
We finally use the global moment identity \eqref{eq:global-moment-start}
to describe the random symbol space $(\mathcal{M}(p,q;\omega))_\star$.

\medskip
\noindent\emph{Inclusion $H(2,q;\omega_\mathrm{r})\subseteq (\mathcal{M}(p,q;\omega))_\star$.}
Let $f\in H(2,q;\omega_\mathrm{r})$ be arbitrary.
Then $\|f\|_{H(2,q;\omega_\mathrm{r})}<\infty$, and the global moment identity
\eqref{eq:global-moment-start} implies in particular that
\[
\mathbb{E}\,\|\mathcal{R}_X f\|_{\mathcal{M}(p,q;\omega)}^{\,q}
  \;\lesssim_{p,q}\; \|f\|_{H(2,q;\omega_\mathrm{r})}^{\,q}
  \;<\;\infty.
\]
By Lemma~\ref{lem:moment-criteria} (applied with $t=q$), this is equivalent to
$\mathcal{R}_X f\in \mathcal{M}(p,q;\omega)$ almost surely.  Hence $f\in(\mathcal{M}(p,q;\omega))_\star$.

\medskip
\noindent\emph{Inclusion $(\mathcal{M}(p,q;\omega))_\star\subseteq H(2,q;\omega_\mathrm{r})$.}
Conversely, suppose $f\in(\mathcal{M}(p,q;\omega))_\star$.  By definition, this means that
$\mathcal{R}_X f\in \mathcal{M}(p,q;\omega)$ almost surely.
Invoking again Lemma~\ref{lem:moment-criteria}, we deduce that there exists
some $t>0$ (and in particular $t=q$) such that
\[
\mathbb{E}\,\|\mathcal{R}_X f\|_{\mathcal{M}(p,q;\omega)}^{\,q}<\infty.
\]
Thus $f\in H(2,q;\omega_\mathrm{r})$.

\medskip
\noindent Now the proof  of Theorem~\ref{thm:H-pq} is complete.
\end{proof}

 \section{Proof of Theorem \ref{thm:tent}}\label{Sec:proof-tent}

\begin{proof}
For each Whitney square $Q_j$ we define a finite localized measure
\[
d\mu_j(z) \;=\; \omega(z)\,\mathbf{1}_{Q_j}(z)\,dA(z).
\]
 On $Q_j$ we then consider
the random local $L^p$-quantity
\[
Y_j
  \;:=\; \Bigg(\int_{Q_j}\big|\mathcal{R}_X f(z)\big|^p\,d\mu_j(z)\Bigg)^{1/p}
  \;=\; \Bigg(\int_{Q_j}\big|\mathcal{R}_X f(z)\big|^p\,\omega(z)\,dA(z)\Bigg)^{1/p}.
\]
By the classical scalar Khintchine–Kahane inequality (for the chosen standard
model and exponent $p$), we have
\[
\mathbb{E}\,|\mathcal{R}_X f(z)|^p
  \;\asymp_p\; \Bigg(\sum_{n\ge0}|a_n|^2|z|^{2n}\Bigg)^{p/2}
  = M_2(f,|z|)^{\,p}.
\]
As a consequence,
\begin{equation}\label{eq:local-p-moment-thm1.10}
\mathbb{E}\,Y_j^{\,p}
  \;=\; \int_{Q_j}\mathbb{E}\,|\mathcal{R}_X f(z)|^p\,d\mu_j(z)
  \;\asymp_{p}\; \int_{Q_j} M_2(f,|z|)^{\,p}\,d\mu_j(z).
\end{equation}
To pass from $\mathbb{E}\,Y_j^{\,p}$ to $\mathbb{E}\,Y_j^{\,q}$ for arbitrary $q>0$, we require a
vector-valued Khintchine–Kahane type result in quasi-Banach spaces. For
convenience, we recall the following lemma.

\begin{lemma}[\cite{IMRN}]\label{lemma:F-Kahane-Khintchine}
Let $\{e_n\}_{n\ge1}$ be a sequence of elements in a $p$-Banach space $\mathcal{X}$,
and let $\{X_n\}_{n\ge0}$ be a standard random sequence. Suppose that
\[
S := \sum_{n=1}^\infty X_n e_n
\]
converges almost surely in $\mathcal{X}$. Then $S\in L^q(\Omega;\mathcal{X})$ for all
$0<q<\infty$, and moreover
\begin{equation}\label{eq:5}
\|S\|_{L^{q_1}(\Omega;\mathcal{X})}
  \;\sim\; \|S\|_{L^{q_2}(\Omega;\mathcal{X})}
\end{equation}
for any $0<q_1,q_2<\infty$, where
\[
\|S\|_{L^q(\Omega;\mathcal{X})}^q
  := \mathbb{E}\big(\|S\|_{\mathcal{X}}^q\big).
\]
\end{lemma}
\noindent In our setting, for each fixed $j$, the random function
$
z\mapsto \mathcal{R}_X f(z)
$ takes values in the quasi-Banach space $\mathcal{X} = L^p(\mu_j)$, which is a
$p$-Banach space for $0<p<1$ and a Banach space if $p\ge1$. 
Thus, Lemma~\ref{lemma:F-Kahane-Khintchine} applied with $\mathcal{X}=L^p(\mu_j)$
and $S = \mathcal{R}_X f(\cdot)$ implies that for every $q>0$,
\begin{equation}\label{eq:local-moment-comparison-thm1.10}
\big(\mathbb{E}\,Y_j^{\,q}\big)^{1/q}
  \;\asymp_{p,q}\; \big(\mathbb{E}\,Y_j^{\,p}\big)^{1/p}.
\end{equation}
Combining \eqref{eq:local-p-moment-thm1.10} and 
\eqref{eq:local-moment-comparison-thm1.10}, we obtain 
\begin{equation}\label{tent star}
   \sum_j \mathbb{E}\,Y_j^{\,q}
  \;\asymp_{p,q}\; \sum_j
  \Bigg(\int_{Q_j} M_2(f,|z|)^{\,p}\,\omega(z)\,dA(z)\Bigg)^{\!q/p}. 
\end{equation}
Finally, we invoke the general moment criterion (Lemma~\ref{lem:moment-criteria})
for random series in F-spaces, applied to the F-space
$
X = A(p,q;\omega).
$
Combining this with \eqref{tent star} (taking $t=q$) gives
\[
\mathbb{P}\big(\mathcal{R}_X f\in A(p,q;\omega)\big)=1
\quad\Longleftrightarrow\quad
\|M_2(f,\cdot)\|_{T(p,q;\omega)}^{\,q}<\infty.
\]
This is precisely the statement of Theorem~\ref{thm:tent}, and the proof is complete.
\end{proof}

\section{Proof of Theorem \ref{thm:N-iff}}\label{N-iff}

\begin{proof}
We first prove the sufficiency part.
Fix a function $g\in L^1(\mu)$ and a level $\lambda>0$, and set the level set
\[
E_\lambda := \big\{z\in\D : \mathcal{N}_\mu^\delta g(z)>\lambda\big\}.
\]
By definition of $\mathcal{N}_\mu^\delta$, for every $z\in E_\lambda$, there
exists at least one dyadic arc $I\in\mathcal{D}$ such that $z\in S(I)$ and
\begin{equation}\label{lambda level}
    \frac{1}{\mu\big(T_\delta(I)\big)}
  \int_{T_\delta(I)^{\ast}} |g|\,d\mu
  \;>\; \lambda.
\end{equation}
We now perform a standard Calderón--Zygmund (CZ) stopping-time selection of dyadic arcs.
Among all dyadic arcs $I$ 
satisfying the above inequality, we select a sub-collection $\{I_j\}_j$ of pairwise disjoint 
\emph{maximal} arcs with respect to inclusion.
 The associated Carleson boxes $\{S(I_j)\}_j$ cover the level set:
\begin{equation}\label{covering}
  E_\lambda = \bigcup_j S(I_j).    
\end{equation}
Using the covering property  and the reverse--Carleson condition, we estimate
the measure of the level set:
\begin{align*}
\nu(E_\lambda)
  &= \sum_j \nu\big(S(I_j)\big)
   \;\le\; (1+\kappa)\sum_j \mu\big(T_\delta(I_j)\big).
\end{align*}
By \eqref{lambda level}, together with  the bounded overlap, we obtain that 
\[
\nu(E_\lambda)
  \;\le\; \frac{1+\kappa}{\lambda}\sum_j \int_{T_\delta(I_j)^{\ast}} |g|\,d\mu
  \;\lesssim\; \frac{1+\kappa}{\lambda}\int_{\D} |g|\,d\mu.
\]
The implicit constant in $\lesssim$ depends only on the overlap parameter. This shows that
 $\mathcal{N}_\mu^\delta$ is weak type $(1,1)$.

\medno Next, we prove the necessity part. 
Assume that the top maximal operator $\mathcal{N}_\mu^\delta$ is of
weak type $(1,1)$; that is, there is a constant $C>0$ such that for every
$g\in L^1(\mu)$ and every $\lambda>0$,
\begin{equation}\label{eq:weak11-max}
\nu\big(\{z\in\D : \mathcal{N}_\mu^\delta g(z)>\lambda\}\big)
  \;\le\; \frac{C}{\lambda}\int_{\D} |g|\,d\mu.
\end{equation}
Fix an arbitrary dyadic arc $I\in\mathcal{D}$ and consider the test function
\[
g := \mathbf{1}_{T_\delta(I)^{\ast}}.
\]
By definition of $\mathcal{N}_\mu^\delta$, for every $z\in S(I)$, we have 
\[
\big(\mathcal{N}_\mu^\delta g\big)(z)
  \;\ge\; \frac{1}{\mu\big(T_\delta(I)\big)}\int_{T_\delta(I)^{\ast}} |g|\,d\mu
  \;=\; \frac{\mu\big(T_\delta(I)^{\ast}\big)}{\mu\big(T_\delta(I)\big)}.
\]
Now fix $\epsilon>0$ and set
\[
\lambda :
  = (1-\epsilon)\,\frac{\mu\big(T_\delta(I)^{\ast}\big)}{\mu\big(T_\delta(I)\big)}.
\]
Hence
\[
S(I) \subseteq E_\lambda
  := \big\{z\in\D : \mathcal{N}_\mu^\delta g(z)>\lambda\big\}.
\]
Applying the weak $(1,1)$ estimate \eqref{eq:weak11-max} to this $g$ and $\lambda$,
we obtain there exists $C>1$ such that
\[
\nu\big(S(I)\big)
  \;\le\; \nu(E_\lambda)
  \;\le\; \frac{C}{\lambda} \int_{\D} |g|\,d\mu
  \;=\; \frac{C}{\lambda}\,\mu\big(T_\delta(I)^{\ast}\big)\leq \frac{C}{1-\epsilon}\,\mu\big(T_\delta(I)\big).
\]
Writing $C = 1+\kappa$ for some $\kappa>0$, we obtain
\[
\nu\big(S(I)\big)
  \;\lesssim\; (1+\kappa)\,\mu\big(T_\delta(I)\big),
\]
which is precisely the reverse–Carleson tail condition $\mathrm{RC}^\delta$
(up to equivalent constants).
This completes the proof of Theorem~\ref{thm:N-iff}.
\end{proof}

\section{Proof of Theorem \ref{thm:BHO}}
\label{UBD}

\noindent We now turn to the technical heart of the paper, where the central challenges are confronted.

\begin{proof}
Throughout the proof we fix $0<r_0<r_h$. The standing assumption
$\omega\in\mathrm{BHO}(r_h)$ means that there exists $\Lambda$ such that, for every $z\in\D$, 
\[
\operatorname*{ess\,sup}_{B_{\rho_h}(z,r_h)}\omega
  \;\le\; \Lambda\,
  \operatorname*{ess\,inf}_{B_{\rho_h}(z,r_h)}\omega.
\]
In particular, the same inequality holds with $r_0<r_h$ in place of $r_h$
(with the same $\Lambda$, or after increasing it once and for all).
Let $f$ be analytic and $q>0$. By $\omega \in \mathrm{BHO}(r_h)$, we have 
\begin{align*}
    |f(w)|^q\ & \lesssim \frac{1}{A(B_{\rho_h}(w,r_0))}\int_{B_{\rho_h}(w,r_0)} |f|^q\,dA \\
    &\leq \frac{1}{A(B_{\rho_h}(w,r_0))} \left(\operatorname*{ess\,inf}_{\zeta\in B_{\rho_h}(w,r_0)}\omega(\zeta)\right)^{-1} \int_{B_{\rho_h}(w,r_0)} |f|^q \omega(z) \,dA \\
    &\leq \Lambda \frac{1}{\mu(B_{\rho_h}(w,r_0))}\int_{B_{\rho_h}(w,r_0)} |f|^q\,d\mu. 
\end{align*}

\medno Our task is now to relate the local ball averages
to the averages over dyadic tops $T_\delta(I)$ and their enlargements.
We first show that $T_\delta(I)$ has uniformly bounded hyperbolic diameter,
with a bound depending only on $\delta$.
Take two arbitrary points $\zeta_1,\zeta_2\in T_\delta(I)$.
We connect them by a piecewise smooth curve composed of three segments:
\begin{itemize}
    \item[(i)] a radial segment from $\zeta_1$ to some intermediate point
$r e^{i\theta_1}$ with $$r\in(1-|I|,1-\delta|I|),$$
    \item[(ii)] an angular segment at radius $r$ from $re^{i\theta_1}$ to $re^{i\theta_2}$, and
    \item[(iii)] a radial segment from $re^{i\theta_2}$ to $\zeta_2$.
\end{itemize}
We estimate the hyperbolic length of each part.
 Each of the two radial legs has length at most $2\log(\delta^{-1})$.
Along the angular part at fixed radius $r$, the line element is at most 
$ \frac{2}{\delta}.
$
Therefore
\[
\operatorname{diam}_{\rho_h}\big(T_\delta(I)\big)
 \;\le\; 4 \log(\delta^{-1}) + \frac{2}{\delta}
 \;:=\; r_{\mathrm{above}}(\delta).
\]
Let $\theta_I$ denote the midpoint of the arc $I$, and define the “center” of
$T_\delta(I)$ by
\[
r_c := 1-\frac{1+\delta}{2}\,|I|,
\qquad
z_I := r_c e^{i\theta_I}\in T_\delta(I).
\]
Then $\zeta\in T_\delta(I)$ implies $\rho_h(\zeta,z_I)\le r_{\mathrm{above}}(\delta)$,
so in particular
\begin{equation}\label{B_above}
    T_\delta(I)\subseteq B_{\rho_h}\big(z_I,r_{\mathrm{above}}(\delta)\big).
\end{equation}
For later use we consider the slightly enlarged ball
\[
T_\delta(I)^\ast
   := B_{\rho_h}\big(z_I,r_0+r_{\mathrm{above}}(\delta)\big).
\]
Then for every $w\in T_\delta(I)$ we have
\begin{equation}\label{T control}
   B_{\rho_h}(w,r_0)\subseteq T_\delta(I)^\ast. 
\end{equation}

\smallno Next, we establish a structural decomposition of the family of dyadic tops into finitely many sub-collections, each consisting of pairwise disjoint elements.
Let $\mathfrak{T}:=\{T_\delta(I): I\in\mathcal{D}\}$ be the family of dyadic
tops for a fixed $0<\delta<1$. We briefly recall the geometric fact that
$\mathfrak{T}$ has uniformly bounded overlap.
By the definition of the reverse Carleson condition, we may assume
$
0<\delta<\frac12
$
without loss of generality.  A straightforward count (using that
the lengths are dyadic) shows that each individual top can meet at most
\[
\operatorname{Ovl}(\delta)
   := 2\Big\lceil\frac{\log(\delta^{-1})}{\log 2}\Big\rceil +2
\]
other tops in $\mathfrak{T}$. 

\medno It is convenient to encode this as a simple graph $G$:
each vertex of $G$ corresponds to a top $T_\delta(I)$;
we join two vertices by an edge if and only if the corresponding tops intersect in $\D$.
By the above discussion, the maximum degree of $G$ satisfies
\[
\Delta(G)\le \operatorname{Ovl}(\delta).
\]
A classical graph coloring theorem (e.g. Brooks’ theorem in this simple form)
asserts that any countable graph with maximum degree $\Delta$ can be properly
colored with at most $\Delta+1$ colors. Therefore, $G$ admits a proper coloring
with at most $\operatorname{Ovl}(\delta)+1$ colors. Equivalently, we can
partition $\mathfrak{T}$ into
\[
\mathfrak{T}
 = \bigcup_{k=1}^{\operatorname{Ovl}(\delta)+1}\mathfrak{B}_k,
\]
such that within each subfamily $\mathfrak{B}_k$ the tops are pairwise disjoint.
We denote by $\mathcal{D}_k\subseteq\mathcal{D}$ the corresponding collection
of arcs, so that
\[
\mathfrak{B}_k=\{T_\delta(I): I\in\mathcal{D}_k\}.
\]

\medno
We next show that each top contains a fixed hyperbolic ball centered at $z_I$;
this will be crucial for measure comparability.
We claim that there exists a radius $r_{\mathrm{below}}(\delta)>0$, depending only on $\delta$,
such that
\begin{equation}\label{eq:ball-comparability}
B_{\rho_h}\big(z_I,r_{\mathrm{below}}(\delta)\big)\subseteq T_\delta(I),
\qquad I\in\mathcal{D}.
\end{equation}
To estimate the hyperbolic distance from $z_I$ to the boundary of $T_\delta(I)$,
we consider
\[
\operatorname{dist}_{\rho_h}\big(z_I, \partial T_\delta(I)\big)
 := \inf_{w \in \partial T_\delta(I)} \rho_h(z_I, w).
\]
Using geometric estimates on the tops, one
can find positive constants $C_1(\delta),C_2(\delta)$, depending only on $\delta$,
such that for any boundary point $w\in\partial T_\delta(I)$,
\[
|z_I-w|\ \geq\ C_1(\delta)\,|I|,
\qquad
|1-\overline{z_I} w|\ \leq\ C_2(\delta)\,|I|.
\]
So for $w\in\partial T_\delta(I)$, we obtain
\[
\rho_h(z_I,w)
  \;\ge\; \operatorname{arctanh}\frac{C_1(\delta)}{C_2(\delta)}:=r_{\mathrm{below}}(\delta).
\]
Consequently,
$
B_{\rho_h}\big(z_I, r_{\mathrm{below}}(\delta)\big)
   \subseteq T_\delta(I),
$
which proves \eqref{eq:ball-comparability}.

\medno 
We now combine the disjointness of each family $\mathfrak{B}_k$ with the ball
inclusion \eqref{eq:ball-comparability}. Fix $x\in\D$ and $1\le k\le
\operatorname{Ovl}(\delta)+1$, and define
\[
\mathcal{H}_k(x)
 := \{\, T\in \mathfrak{B}_k : x\in T^\ast \,\}.
\]
By the definition of $T_\delta(I)^\ast$ and $z_I$, we have
\[
\#\mathcal{H}_k(x)
 = \#\{\, I\in \mathcal{D}_k:\ x\in B_{\rho_h}(z_I,r_0+r_{\mathrm{above}}(\delta)) \,\}.
\]
If $x\in T_\delta(I)^\ast=B_{\rho_h}(z_I,r_0+r_{\mathrm{above}}(\delta))$, then
by the triangle inequality and \eqref{eq:ball-comparability},
\[
B_{\rho_h}(z_I,r_{\mathrm{below}}(\delta))
  \subseteq B_{\rho_h}\big(x, r_{\mathrm{below}}(\delta)+r_0+r_{\mathrm{above}}(\delta)\big)
  =: B_{\rho_h}\big(x,R(\delta,r_0)\big).
\]
Moreover, within each $\mathfrak{B}_k$ the balls
$B_{\rho_h}(z_I,r_{\mathrm{below}}(\delta))$ are pairwise disjoint.
We recall that for any $z\in\D$ and $s>0$,
\begin{equation}\label{eq:area-ball-final}
\operatorname{Area}_h\big(B_{\rho_h}(z,s)\big)
 = 4\pi\sinh^2\!\Big(\frac{s}{2}\Big),
\end{equation}
and this quantity is independent of the center $z$ because automorphisms of
$\D$ are hyperbolic isometries. Hence the disjointness gives
\[
\#\mathcal{H}_k(x)\cdot 4\pi\sinh^2\!\Big(\frac{r_{\mathrm{below}}(\delta)}{2} \Big)
 \;\le\; 4\pi\sinh^2\!\Big(\frac{R(\delta,r_0)}{2}\Big).
\]
Thus there exists a constant $N(\delta,r_0)$ such that
\[
\sup_{x\in\D}
  \sum_{T\in\mathfrak{B}_k}\mathbf{1}_{T^\ast}(x)
  \;\le\; N(\delta,r_0),
\]
for every $1\le k\le\operatorname{Ovl}(\delta)+1$.
Summing over $k$, we conclude that the family $\mathfrak{T}$ of tops is
$(\operatorname{Ovl}(\delta)+1)\,N(\delta,r_0)$–overlapped:
\[
\sup_{x\in\D}
  \sum_{I\in\mathcal{D}}\mathbf{1}_{T_\delta(I)^\ast}(x)
  \;\le\; (\operatorname{Ovl}(\delta)+1)\,N(\delta,r_0).
\]

\smallno 
We now show that there exists $C_0=C_0(r_0,\delta,\Lambda,r_h)$ such that for
every dyadic arc $I$ and every $w\in T_\delta(I)$,
\begin{equation}\label{eq:top-ball-comparable}
\mu\big(T_\delta(I)\big)
 \;\le\; C_0\,\mu\big(B_{\rho_h}(w,r_0)\big).
\end{equation}
 Fix $I$ and $w\in T_\delta(I)$. From \eqref{B_above} we know that
$T_\delta(I)\subset B_{\rho_h}(w,r_0+r_{\mathrm{above}}(\delta))=:B_{\mathrm{above}}$. 
To propagate the  bounded hyperbolic oscillation condition  from $B_{\rho_h}(w,r_0)$ to the whole set
$B_{\mathrm{above}}$, we construct chains of overlapping balls of radius $r_0$
joining arbitrary points of $B_{\mathrm{above}}$.
Any two points
$\zeta_1,\zeta_2\in B_{\mathrm{above}}$ can be joined by a hyperbolic geodesic
of length at most $2r_0+2r_{\mathrm{above}}(\delta)$. Choose points
$\xi_1,\dots,\xi_m$ on this geodesic so that
\[
\rho_h(\zeta_1,\xi_1)<r_0,\qquad
\rho_h(\xi_j,\xi_{j+1})< r_0\ (1\le j<m),\qquad
\rho_h(\xi_m,\zeta_2)<r_0,
\]
with
\[
m\le N_0:=1+\Big\lceil\frac{2r_0+2r_{\mathrm{above}}(\delta)}{r_0}\Big\rceil.
\]

\smallno For a measurable set $E$ with positive area, let
\[
\langle\omega\rangle_{E}
 :=\frac{1}{A(E)}\int_E \omega\,dA
\]
denote the average of $\omega$ over $E$. 
Now suppose $B'$ and $B''$ are two radius-$r_0$ balls with nonempty
intersection. Pick $\eta\in B'\cap B''$ such that 
\[
\langle\omega\rangle_{B'}
  \;\le\; \Lambda \,\omega(\eta)
  \;\le\; \Lambda^2\,\langle\omega\rangle_{B''}.
\]
Iterating along the chain of $m$ balls
$B_{\rho_h}(\xi_1,r_0),\dots,B_{\rho_h}(\xi_m,r_0)$,
we obtain 
\[
\omega(\zeta_1)
  \;\le\; \Lambda\,\langle\omega\rangle_{B_{\rho_h}(\xi_1,r_0)}
  \;\le\; \Lambda^{2m-1}\,
       \langle\omega\rangle_{B_{\rho_h}(\xi_m,r_0)}
  \;\le\; \Lambda^{2m}\,\omega(\zeta_2),
\]
which gives
\[
\omega(\zeta_1)\ \le\ \Lambda^{2N_0}\,\omega(\zeta_2),
\qquad\text{for a.e.\ }\zeta_1,\zeta_2\in B_{\mathrm{above}}.
\]
Taking essential supremum over $\zeta_1$ and essential infimum over $\zeta_2$ 
yields
\begin{equation}    \label{BHO on above}
\operatorname*{ess\,sup}_{B_{\mathrm{above}}}\omega
  \;\le\; \Lambda^{2N_0}\,
       \operatorname*{ess\,inf}_{B_{\mathrm{above}}}\omega.
\end{equation}
The inequality \eqref{BHO on above} gives
\[
\mu(B_{\mathrm{above}})
 \;\le\; \Lambda^{2N_0}\,
         \frac{A(B_{\mathrm{above}})}{A(B_{\rho_h}(w,r_0))}
         \,\mu\big(B_{\rho_h}(w,r_0)\big).
\]
It remains to estimate the geometric factor
$A(B_{\mathrm{above}})/A(B_{\rho_h}(w,r_0))$.
Write $s=\tanh(r_0/2)$ and $S=\tanh((r_0+r_{\mathrm{above}}(\delta))/2)$.
By comparing Euclidean balls with hyperbolic balls, one finds
\[
\frac{A(B_{\mathrm{above}})}{A(B_{\rho_h}(w,r_0))}
 = \frac{S^2}{s^2}\,
   \frac{(1-|w|^2 s^2)^2}{(1-|w|^2 S^2)^2}.
\]
For $u:=|w|^2\in[0,1)$, the function
\[
f(u)
 := \frac{(1-us^2)^2}{(1-uS^2)^2}
\]
is increasing, so the worst case is $u\uparrow 1^-$.
Letting $u\to 1$ and using the trigonometric inequalities,
we obtain
\[
\sup_{w\in\D}
 \frac{A(B_{\mathrm{above}})}{A(B_{\rho_h}(w,r_0))}
 = \frac{S^2}{s^2}\cdot\frac{(1-s^2)^2}{(1-S^2)^2}
 = \left(\frac{\sinh(r_0+r_{\mathrm{above}}(\delta))}{\sinh r_0}\right)^{2}
 \;\lesssim_{r_0,\delta}\; 1.
\]
More concretely, one may bound it by
\[
\left(\Big\lceil\frac{r_0+r_{\mathrm{above}}(\delta)}{r_0}\Big\rceil\cosh r_0\right)^2.
\]
This implies \eqref{eq:top-ball-comparable}
holds 
for some constant $C_0=C_0(r_0,\delta,\Lambda,r_h)$. Arguing in the same way and using \eqref{eq:ball-comparability}, we also obtain a constant 
\(C_1 = C_1(r_0,\delta,\Lambda,r_h)\) such that \begin{equation}\label{eq:linfty}
    \mu(T(I)^\ast)\leq C_1\mu(T(I)).
\end{equation}
We now demonstrate that the radial maximal function is pointwise controlled by the top maximal operator, thereby completing the proof.
Combining \eqref{T control} with \eqref{eq:top-ball-comparable},
for every $w\in T_\delta(I)$, we have
\[
|f(w)|^q
 \;\lesssim\; \frac{1}{\mu(B_{\rho_h}(w,r_0))}
              \int_{B_{\rho_h}(w,r_0)} |f|^q\,d\mu
 \;\le\; C_0\,\frac{1}{\mu(T_\delta(I))}
              \int_{T_\delta(I)^\ast} |f|^q\,d\mu.
\]
Thus, for each $z\in\D$,
\begin{align*}
\mathfrak{D}(f)(z)^q
 &:= \sup_{0<r<1}|f(rz)|^q \\[4pt]
 &\le \sup\big\{|f(w)|^q:\ w=rz,\ 0<r<1\big\} \\[4pt]
 &\lesssim \sup\Big\{
    \frac{1}{\mu(T_\delta(I))}
    \int_{T_\delta(I)^\ast} |f|^q\,d\mu
    : I\in\mathcal{D},\ w\in T_\delta(I),\ w=rz,\ 0<r<1\Big\} \\[4pt]
 &= \sup_{\substack{I\in\mathcal{D}\\ z\in S(I)}}
    \frac{1}{\mu(T_\delta(I))}
    \int_{T_\delta(I)^\ast} |f|^q\,d\mu
  \;=:\; \mathcal{N}(|f|^q)(z),
\end{align*}
where $\mathcal{N}$ is the top maximal operator associated with the family
$\{T_\delta(I)\}$.
Fix $0<q<p$. Then
\[
\|\mathfrak{D}(f)\|_{L^p(\mu)}^q
  = \|\mathfrak{D}(f)^q\|_{L^{p/q}(\mu)}
 \;\lesssim\; \|\mathcal{N}(|f|^q)\|_{L^{p/q}(\mu)}.
\] 
By combining \eqref{eq:linfty} with Theorem~\ref{thm:N-iff} and the Marcinkiewicz interpolation theorem, the operator $\mathcal{N}$ is bounded on
$L^{p/q}(\mu)$ under the reverse Carleson assumptions, and hence
\[
\|\mathcal{N}(|f|^q)\|_{L^{p/q}(\mu)}
 \;\lesssim\; \||f|^q\|_{L^{p/q}(\mu)}
 = \|f\|_{L^p(\mu)}^q.
\]
Therefore
\[
\|\mathfrak{D}(f)\|_{L^p(\mu)}
 \;\lesssim\; \|f\|_{L^p(\mu)}.
\]
Finally, for fixed $r\in(0,1)$ and $z\in\D$,
\[
|T_r f(z)| = |f(r z)| \le \mathfrak{D}(f)(z),
\]
so
\[
\|T_r f\|_{L^p(\mu)}
 \le \|\mathfrak{D}(f)\|_{L^p(\mu)}
 \lesssim \|f\|_{L^p(\mu)},
\]
and the bound is uniform in $0<r<1$. Hence
\[
\sup_{0<r<1}\|T_r\|_{A_\omega^p\to A_\omega^p} <\infty,
\]
which is exactly the desired conclusion of Theorem~\ref{thm:BHO}.
This completes the proof.
\end{proof}

\section{Proof of Theorem \ref{Hardy-Littlewood maiximal operator}} 
\label{Hardy-maximal}

\begin{proof}
As in  the proof of Theorem~\ref{thm:BHO}, there exists a constant
$C_{\mathrm{loc}}>0$ (depending only on $r_1,r_0,\Lambda,r_h$)
such that for all $z\in\D$ and all $s$ with $r_1<s\le r_0$,
\begin{equation}\label{eq:local-doubling-5s}
\mu\!\big(B_{\rho_h}(z,5s)\big)
 \;\le\; C_{\mathrm{loc}}\,
        \mu\!\big(B_{\rho_h}(z,s)\big).
\end{equation}
To obtain a weak $(1,1)$ estimate, we fix $\lambda>0$ and consider the level set
\[
E_\lambda
 := \big\{z\in\D: M_{r_1,r_0}g(z)>\lambda\big\}.
\]
The family
$\mathcal{B}:=\{B_{\rho_h}(z,s_z): z\in E_\lambda\}$ forms a covering of $E_\lambda$.
By the Vitali selection principle for hyperbolic balls (stated and proved in
Appendix~\ref{sec:Vitali of hyperbolic ball}), there exists a countable
subfamily of pairwise disjoint balls
$
\{B_j\}_j
   := \{B_{\rho_h}(z_j,s_{z_j})\}_j
   \subseteq \mathcal{B},
$
such that

\[
E_\lambda\ \subseteq\ \bigcup_j 5 B_j
\qquad\text{and}\qquad
\sum_j \mathbf{1}_{5B_j}\ \le\ C(r_0,r_1)\quad\text{pointwise.}
\]
Applying the fixed-scale local doubling property \eqref{eq:local-doubling-5s} gives
\[
\mu(E_\lambda)
  \;\le\; \sum_j \mu(5B_j)
  \;\le\; C_{\mathrm{loc}} \sum_j \mu(B_j).
\]
Summing over $j$ and using that the balls $B_j$ are pairwise disjoint, we obtain
\begin{equation}\label{weak 11}
    \sum_j \mu(B_j)
 \;\le\; \frac{1}{\lambda}\sum_j\int_{B_j}|g|\,d\mu
 \;\le\; \frac{1}{\lambda}\int_{\D} |g|\,d\mu
 = \frac{1}{\lambda}\,\|g\|_{L^1(\mu)},
\end{equation}
which is the desired weak $(1,1)$ estimate for the operator $M_{r_1,r_0}$.

\medno 
Note that 
\begin{equation}\label{eq:Linfty-bound}
\|M_{r_1,r_0}g\|_{L^\infty(\mu)}
  \le \|g\|_{L^\infty(\mu)}.
\end{equation}
 The Marcinkiewicz interpolation theorem applies and yields that
for every $1<p<\infty$,
\[
\|M_{r_1,r_0}g\|_{L^p(\mu)}
  \le C(p,r_1,r_0,\Lambda,r_h)\,\|g\|_{L^p(\mu)}.
\]
This establishes the desired strong $(p,p)$ bound  and completes the proof of Theorem
\ref{Hardy-Littlewood maiximal operator}.
\end{proof}

\section{Proof of Theorem \ref{thm:WSMP-LD}}\label{Sec:Carleson-embedding-equivalence}

\begin{proof}
By an argument analogous to that used in the proof of Theorem~\ref{thm:BHO}, for a fixed hyperbolic scale $r_0>0$, one can construct a dyadic grid $\{Q\}$ together with an associated family of dilated sets $\{Q^{*}\}$ satisfying the following properties:
\begin{enumerate}
    \item [(i)] There exists a constant $\kappa=\kappa\left(r_0\right) \geq 1$ such that the "ball-box sandwich" holds:
    \begin{equation}\label{eq:sanwich}
    B_{\rho_h}\left(z, r_0\right) \subseteq Q^* \subseteq B_{\rho_h}\left(z, \kappa r_0\right) \quad \text { whenever } z \in Q .
    \end{equation}

\item [(ii)] The family $\left\{Q^*\right\}$ has uniformly bounded overlap:
$$
\sum_Q \mathbf{1}_{Q^*} \leq N\left(r_0\right) \quad \text { pointwise on } \mathbb{D} .
$$
\end{enumerate}
Fix a box $Q$ from the grid. For any $z \in Q$,
from \eqref{eq:WSMP},
\begin{equation}\label{ave}
    |f(z)|^{p}\ \le\ \ C_{\mathrm{ws}}\ \frac{1}{\mu(B_{\rho_h}\left(z, r_0\right))}\int_{B_{\rho_h}\left(z,r_0\right)}|f|^{p}\,d\mu.
\end{equation}
Using \eqref{eq:sanwich} and the local doubling property \eqref{eq:LD}, we obtain
\[
\frac{1}{\mu(B_{\rho_h}\left(z, r_0\right))}\int_{B_{\rho_h}\left(z, r_0\right)}|f|^{p}\,d\mu \ \le\ D_\kappa\ \frac{1}{\mu(Q^\ast)}\int_{Q^\ast}|f|^{p}\,d\mu.
\]
Combining this inequality with \eqref{ave} yields the boxwise estimate
\begin{equation}\label{eq:local-SA}
\sup_{z\in Q}|f(z)|^{p}\ \le\ C_{\mathrm{ws}} \ D_\kappa\ \frac{1}{\mu(Q^\ast)}\int_{Q^\ast}|f|^{p}\,d\mu.
\end{equation}

\medno 
We now decompose the integral over $\D$:
\[
\int_{\D}|f|^{p}\,d\nu\ =\ \sum_{Q}\ \int_{Q}|f|^{p}\,d\nu\ \le\ \sum_{Q}\nu(Q)\,\sup_{z\in Q}|f(z)|^{p}.
\]
Applying \eqref{eq:local-SA} to the supremum and using the measure comparison  $$\nu(Q)\le C\,\mu(Q)\le C\,\mu(Q^\ast),$$ we obtain
\[
\int_{\D}|f|^{p}\,d\nu\ \le\ C_{\mathrm{ws}} \,D_\kappa\,C\ \sum_{Q}\int_{Q^\ast}|f|^{p}\,d\mu
\ =\ C_{\mathrm{ws}} \,D_\kappa\,C\ \int_{\D}\Big(\sum_{Q}\mathbf{1}_{Q^\ast}\Big)|f|^{p}\,d\mu.
\]
Finally, applying the bounded overlap property (ii), we conclude that
\[
\int_{\D}|f|^{p}\,d\nu\ \le\ C_{\mathrm{ws}} \,D_\kappa\,N(r_0)\,C\ \int_{\D}|f|^{p}\,d\mu.
\]
\end{proof}

\bignobf{Concluding remarks.} 
In this paper we consider two problems related to the UBD condition on non-radially weighted Bergman spaces. For the first one, namely the Littlewood problem, UBD is assumed and it enables us to borrow the machinery in \cite{IMRN}. Then the second problem in this paper is to provide a concrete, geometric sufficient condition for UBD.
The characterization of weights such that UBD holds remains an elusive open problem.

\section{Appendix}
\subsection{Equivalence of Whitney decompositions}\label{subsec:Equivalence-Whitney-Decompositions}

\begin{definition}[The Whitney Decomposition]\label{Def:Whitney-Decomposition}
   This is a partition of $\mathbb{D}$ into a collection of ``squares" $\{Q_{j}\}$ whose interiors are disjoint
and which satisfy:
\begin{itemize}
    \item $\cup_{j}\overline{Q_{j}}=\mathbb{D}$,
    \item the hyperbolic diameters of the squares are uniformly bounded,
    \item the Euclidean diameter of a square $Q_{j}$ is comparable to its distance from the boundary, i.e.,
    $diam(Q_{j})\approx dist(Q_{j},\partial\mathbb{D})$,
    \item each point in $\mathbb{D}$ is contained in a uniformly bounded number of ``inflated" squares.
\end{itemize}
\end{definition}

\begin{lemma}[Whitney-equivalence for $T(p,q;\omega)$]\label{lem:Whitney-independence}
Let $0<p,q<\infty$, let $\omega\ge0$ be any measurable weight on $\mathbb D$, and let $\mathcal W_1,\mathcal W_2$ be two Whitney decompositions of $\mathbb D$ with the usual uniform constants\footnote{Each $Q\in\mathcal W_i$ has Euclidean diameter comparable to $\mathrm{dist}(Q,\partial\mathbb D)$, hyperbolic diameter uniformly bounded, interiors pairwise disjoint, $\bigcup_Q\overline Q=\mathbb D$, and for a fixed inflation factor $\lambda>1$, the family $\{\lambda Q:Q\in\mathcal W_i\}$ has uniformly bounded overlap.}. Then 
\[
\|g\|_{T_{\mathcal W_1}(p,q;\omega)}
\asymp
\|g\|_{T_{\mathcal W_2}(p,q;\omega)}\qquad\text{for all measurable }g.
\]
The equivalence constants depend only on $p,q$ and on the Whitney parameters of $\mathcal W_1,\mathcal W_2$ (but \emph{not} on $g$ or on $\omega$). In particular, finiteness of the sum in \eqref{eq:def-weighted-tent} is independent of the chosen Whitney decomposition.
\end{lemma}

\begin{proof}
Write $s:=q/p\in(0,\infty)$. We reduce the proof to two geometric facts about Whitney partitions; both are standard and follow from the defining properties.\medskip

\noindent\emph{Geometric Fact 1 (size comparability on overlaps).} If $Q\in\mathcal W_1$ and $Q'\in\mathcal W_2$ satisfy $Q\cap Q'\neq\varnothing$, then $\mathrm{diam}(Q)\asymp \mathrm{diam}(Q')$ with constants depending only on the Whitney parameters. Consequently, there exists $\lambda\ge1$ (universal) such that $Q'\subseteq \lambda Q$ and $Q\subseteq \lambda Q'$.\medskip

\noindent \emph{Geometric Fact 2 (uniform intersection number).} There exists $N<\infty$ (universal) such that for every $Q\in\mathcal W_1$ the set
\[
\mathcal E(Q):=\{Q'\in\mathcal W_2:\ Q\cap Q'\neq\varnothing\}
\]
has cardinality $\#\mathcal E(Q)\le N$, and symmetrically, for every $Q'\in\mathcal W_2$ the set $\mathcal F(Q'):=\{Q\in\mathcal W_1:\ Q\cap Q'\neq\varnothing\}$ has $\#\mathcal F(Q')\le N$. (This is a packing estimate: by Fact~1, all the $Q'$ intersecting a fixed $Q$ are disjoint sets of comparable size contained in $\lambda Q$, hence their number is bounded by an area ratio.)

\medskip\noindent
We begin by establishing a local estimate. For each $Q\in\mathcal W_1$,
\[
\int_Q |g|^{p}\omega\,dA=\sum_{Q'\in\mathcal E(Q)}\int_{Q\cap Q'} |g|^{p}\omega\,dA.
\]
Hence, using the elementary inequality
\[
\Big(\sum_{k=1}^{m} a_k\Big)^{s}\le
\begin{cases}
m^{\,s-1}\sum_{k=1}^{m} a_k^{s},& s\ge1,\\[2pt]
\sum_{k=1}^{m} a_k^{s},& 0<s\le1,
\end{cases}
\qquad(a_k\ge0),
\]
and the bound $m=\#\mathcal E(Q)\le N$, we get
\[
\Big(\!\int_Q |g|^{p}\omega\,dA\Big)^{\!s}
\le C_s(N)\sum_{Q'\in\mathcal E(Q)}\Big(\!\int_{Q\cap Q'} |g|^{p}\omega\,dA\Big)^{\!s},
\]
with $C_s(N)=\max\{1,\,N^{\,s-1}\}$.

 \medno Summing over $Q\in\mathcal W_1$ and swapping the order of summation yields
\[
\sum_{Q\in\mathcal W_1}\Big(\!\int_Q |g|^{p}\omega\,dA\Big)^{\!s}
\le C_s(N)\sum_{Q'\in\mathcal W_2}\ \sum_{Q\in\mathcal F(Q')}\Big(\!\int_{Q\cap Q'} |g|^{p}\omega\,dA\Big)^{\!s}.
\]
For fixed $Q'$, the family $\{Q\cap Q':\,Q\in\mathcal F(Q')\}$ is pairwise disjoint and forms a partition of $Q'$ (up to a null set). Applying again the elementary inequality, now to the disjoint pieces inside $Q'$, we obtain
\[
\sum_{Q\in\mathcal F(Q')}\Big(\!\int_{Q\cap Q'} |g|^{p}\omega\,dA\Big)^{\!s}
\le D_s(N)\,\Big(\!\int_{Q'} |g|^{p}\omega\,dA\Big)^{\!s},
\]
where $D_s(N)=1$ if $s\ge1$, and $D_s(N)=N^{\,1-s}$ if $0<s\le1$ (use $\sum a_k^{s}\le m^{1-s}(\sum a_k)^{s}$ for $0<s\le 1$ with $m=\#\mathcal F(Q')\le N$). Therefore
\[
\sum_{Q\in\mathcal W_1}\Big(\!\int_Q |g|^{p}\omega\,dA\Big)^{\!s}
\le C_s(N)\,D_s(N)\sum_{Q'\in\mathcal W_2}\Big(\!\int_{Q'} |g|^{p}\omega\,dA\Big)^{\!s}.
\]
This establishes one direction of the equivalence. Reversing the roles of $\mathcal W_1$ and $\mathcal W_2$ yields the converse inequality with an analogous constant. Combining the two bounds yields the desired equivalence with a constant depending only on $s$ and on the Whitney parameters.

\end{proof}

\begin{corollary}[Well-posedness of $T(p,q;\omega)$ and $A(p,q;\omega)$]\label{cor:well-posed}
For $0<p,q<\infty$ and any measurable weight $\omega\ge0$, the definition of $T(p,q;\omega)$ is independent (up to equivalent quasi-norms) of the Whitney decomposition used. In particular, one may write unambiguously
\[
\|g\|_{T(p,q;\omega)}:=\Bigg(\sum_{Q\in\mathcal W}\Big(\int_Q |g|^{p}\omega\,dA\Big)^{\!q/p}\Bigg)^{1/q},
\]
for \emph{any} Whitney decomposition $\mathcal W$. The same holds for the analytic tent space $A(p,q;\omega)=T(p,q;\omega)\cap H(\mathbb D)$.
\end{corollary}

\subsection{Vitali selection principle}\label{sec:Vitali of hyperbolic ball}

\begin{lemma}\label{vitali-covering-of-ball}
Fix $0<r_1<r_2$.  For any family $\mathcal{F}=\{B_{\rho_h}(z_i,r_i)\}_{i\in \mathcal{J}}$ of hyperbolic balls with $r_1<r_i\le r_2$, there exists a pairwise disjoint subfamily $\mathcal{G}\subseteq\mathcal{F}$  and a constant $c_0$ such that:
\begin{itemize}
\item[(i)] $\bigcup_{B\in\mathcal{F}} B \subseteq \bigcup_{G\in\mathcal{G}} c_0 G$.
\item[(ii)]  The point-wise overlap of the subfamily $\{c_0 G:\ G\in\mathcal{G}\}$ is bounded by $N(r_1,r_2)$; that is,
\[
\sup_{z\in\mathbb{D}}\ \sum_{G\in\mathcal{G}}\mathbf{1}_{c_0 G}(z)\ \le\ N(r_1,r_2).
\]
\end{itemize}
\end{lemma}

\begin{proof}

We proceed in two parts: selection and coverage; then scale-wise overlap. 
 Set \( J_0 := \mathcal{F} \).  
   If \( r_j = 0 \) for all \( j \in J_0 \), then stop (and obtain \( S = \varnothing \), satisfying the conclusion).  
   Otherwise, proceed with the following recursive steps.
For \( m = 1, 2, \ldots \), proceed as follows: Define
   $$
   R_m := \sup\{ r_j : j \in J_{m-1} \}.
   $$
   If \( R_m = 0 \), then stop (since all remaining balls have radius 0, this does not affect the covering).
 Since \( R_m > 0 \), by the property of the supremum, there exists an index 
   \( k_m \in J_{m-1} \) such that
   $
   r_{k_m} > \tfrac{1}{2} R_m.
   $
   (If such an index does not exist, then all \( r_j \le \tfrac{1}{2} R_m \), implying \( R_m \le \tfrac{1}{2} R_m \), a contradiction.)
 Add \( k_m \) to the selected set, and define
   $$
   J_m := \{ j \in J_{m-1} : B_j \cap B_{k_m} = \varnothing \},
   $$
   i.e., \( J_m \) consists of those indices whose corresponding balls are disjoint from \( B_{k_m} \).
 If \( J_m = \varnothing \), stop; otherwise continue to the next step \( m + 1 \).
This procedure yields a finite or infinite index sequence \( k_1, k_2, \ldots \), and we define
$
\mathcal{G} := \{ B_{k_m} \}.
$
By maximality, for each $B=B_{\rho_h}(z,r)\in\mathcal{F}$ there exists $G=B_{\rho_h}(w,s)\in\mathcal{G}$ with $B\cap G\neq\emptyset$ and $s\ge r/2$.
First,
if $\zeta\in B$, then by the triangle inequality,
\[
\rho_h(\zeta,w)
\le\rho_h(\zeta,z)+\rho_h(z,w)\le r+(r+s)\le 5s,
\]
where we used $r\le 2s$. Hence $B\subseteq B_{\rho_h}(w,5s)=5G$, so
\[
\bigcup_{B\in\mathcal{F}}B\subseteq\bigcup_{G\in\mathcal{G}}5G.
\]
This proves (i) with $c_0=5$.

\medno Next,
fix $x\in\mathbb{D}$, and consider
\[
\mathcal{H}(x):=\{\,G=B_{\rho_h}(w,s)\in\mathcal{G}:\ x\in5G\,\}.
\]
For each $G\in\mathcal{H}(x)$, $\rho_h(w,x)<5r_2$. The balls in $\mathcal{H}(x)$ are pairwise disjoint. Shrink each to
\[
\widetilde{G}:=B_{\rho_h}\bigl(w,s/5\bigr).
\]
Then the $\widetilde{G}$ remain disjoint and lie inside
\[
B_{\rho_h}\!\left(x,\,\rho_h(x,w)+\tfrac{s}{5}\right)
\subseteq B_{\rho_h}\!\left(x,\,5s+\tfrac{s}{5}\right)
\subseteq B_{\rho_h}\!\left(x,\tfrac{26}{5}r_2\right).
\]
Thus all $\widetilde{G}$ have radii $\ge\tfrac{1}{5}r_1$ and are contained in a ball of radius $\tfrac{26}{5}r_2$.
Finally,
\[
\#\mathcal{H}_k(x)\le
\frac{\sinh^2(\tfrac{13}{5}r_2)}{\sinh^2(r_1/10)}
=:N(r_1,r_2).
\]
Since $x\in5G$ iff $G\in\mathcal{H}_k(x)$, this yields
\[
\sup_{x\in\mathbb{D}}\ \sum_{G\in\mathcal{G}_k}\mathbf{1}_{5G}(x)
\ \le\ N_(r_1,r_2),
\]
proving (ii).  \qedhere
\end{proof}

\mednobf{Remark.} Let $x\in\mathbb{D}$ and choose a unit-speed hyperbolic geodesic ray $\gamma:[0,\infty)\to\mathbb{D}$ with $\gamma(0)=x$. For $k\ge1$, set $r_k=4^{-k}$ and $z_k=\gamma(2r_k)$. Then the balls $B_k:=B_{\rho_h}(z_k,r_k)$ are pairwise disjoint since for $j>k$,
\[
\rho_h(z_k,z_j)=2|r_k-r_j|\ge \tfrac32 r_k> r_k+r_j,
\]
while $x\in B_{\rho_h}(z_k,3r_k)$ for all $k$ because $\rho_h(x,z_k)=2r_k<3r_k$. Hence a priori bounded overlap for fixed dilates fails if radii are allowed to approach $0$. This is why we impose the (already available) structural lower bound $r\ge r_0$.

\normalsize

\bignobf{Acknowledgement}
X. Fang is supported by National Science and Technology Council (NSTC 112-2115-M-008-010-MY2). S. Hou is supported by National Natural 
 Science Foundation (NNSF) of China (No. 11971340).  Q. Zhou is supported by  National Natural Science Foundation (NNSF) of China (No.12501162), China Postdoctoral Science Foundation (No. 2024M762280) and Natural Science Foundation of Jiangsu Province (No. BK20250832).

\bignobf{Author Contributions}
All authors reviewed the manuscript.

\bignobf{Conflict of Interest}
The authors declare no competing interests.

\providecommand{\bysame}{\leavevmode\hbox to3em{\hrulefill}\thinspace}
\providecommand{\MR}{\relax\ifhmode\unskip\space\fi MR }
\providecommand{\MRhref}[2]{%
  \href{http://www.ams.org/mathscinet-getitem?mr=#1}{#2}
}
\providecommand{\href}[2]{#2}

\end{document}